\documentclass[10pt,oneside,letter,reqno]{amsart}

\usepackage[T1]{fontenc}
\usepackage[utf8]{inputenc}
\usepackage[british]{babel}
\usepackage{amsmath, amsthm, amsfonts, amssymb, slashed, stmaryrd}
\usepackage{tikz-cd}
\usepackage[colorlinks=true,allcolors=gray]{hyperref}

\calclayout

\numberwithin{equation}{section}

\theoremstyle{plain}

\newtheorem{thm}{Theorem}[section]
\newtheorem{lem}[thm]{Lemma}
\newtheorem{prop}[thm]{Proposition}

\theoremstyle{definition}

\newtheorem{rem}[thm]{Remark}

\theoremstyle{remark}
\newtheorem{step}{Step}

\newcommand{\Z}{\mathbb{Z}}
\newcommand{\R}{\mathbb{R}}
\newcommand{\C}{\mathbb{C}}
\newcommand{\U}{\mathrm{U}}
\newcommand{\SU}{\mathrm{SU}}
\newcommand{\loc}{\mathrm{loc}}
\newcommand{\del}{\overline{\partial}}
\DeclareMathOperator{\rank}{rank}
\DeclareMathOperator{\im}{im}

\DeclareMathOperator{\Hom}{Hom}

\newcommand\restr[2]{{% wee make the whole thing an ordinary symbol
  \left.\kern-\nulldelimiterspace % automatically resize the bar with \right
  #1 % the function
  \vphantom{\big|} % pretend it's a little taller at normal size
  \right|_{#2} % this is the delimiter
  }}

\begin{document} 
\title{Adiabatic limits and Kazdan--Warner equations}
\author{Aleksander Doan}
%\address{Department of Mathematics, Stony Brook University, John S. Toll Drive, Stony Brook, NY 11794, USA}
%\email{aleksander.doan@stonybrook.edu}
%\date{\today}
\begin{abstract}
We study the limiting behaviour of solutions to abelian vortex equations when the volume of the underlying Riemann surface grows to infinity. We prove that the solutions converge smoothly away from finitely many points. The proof relies on \emph{a priori} estimates for functions satisfying generalised Kazdan--Warner equations. We relate our results to the work of Hong, Jost, and Struwe on classical vortices, and that of Haydys and Walpuski on the Seiberg--Witten equations with multiple spinors.
\end{abstract}
\maketitle

\section{Introduction}
The \emph{vortex equations} on Riemann surfaces originate from the Ginzburg--Landau model of superconductivity. They were first brought to the attention of mathematicians by Jaffe and Taubes \cite{jt}. Since then there has been a considerable body of work aimed at understanding these equations and their many generalisations. The purpose of this article is to study one such generalisation in the context of \emph{adiabatic limits} and \emph{compactifications} in two- and three-dimensional gauge theories. 

Let $\Sigma$ be a closed Riemann surface and $L \to \Sigma$ a Hermitian line bundle. A \emph{(classical) vortex} is a pair of a unitary connection $A$ on $L$ and a section $\varphi$ of $L$ satisfying
\begin{equation}\label{eqn:vortex0}
\left\{
\begin{array}{l}
\del_{A} \varphi = 0, \\
i \Lambda F_{A}    =  1 - | \varphi |^2,
\end{array} \right.
\end{equation}
where  $\Lambda F_A$ is the Hodge dual of the curvature form. The space of vortices modulo gauge equivalence is simply the symmetric product $\mathrm{Sym}^d \Sigma$ with $d = \deg L$. A point in $\mathrm{Sym}^d \Sigma$ corresponds to a degree $d$ effective divisor $D$ and there exists a unique, up to gauge equivalence, solution $(A,\varphi)$ having  $D$ as the zero divisor of $\varphi$ \cite{taubes, noguchi, bradlow, garcia-prada1, garcia-prada2}.

Scaling the metric on $\Sigma$ by a factor $\epsilon^{-1}$ results in the modified equations
\begin{equation}\label{eqn:vortex1}
\left\{
\begin{array}{l}
\del_{A} \varphi = 0, \\
\epsilon^2 i\Lambda F_{A}    = 1 - | \varphi |^2.
 \end{array} \right.
\end{equation}
A question arises about the behaviour of solutions in the limit $\epsilon \to 0$, as the volume of $\Sigma$ grows to infinity. This is the idea of the adiabatic limit which has been used in many contexts in gauge theory and Riemannian geometry \cite{bismut, ds, fine}. The question for the vortex equations was answered by Hong, Jost, and Struwe \cite{hjs} whose result can be stated as follows. Fix a degree $d$ effective divisor $D = \sum_k m_k x_k$, where $m_k \in \Z_{\geq 0}$ and $x_k \in \Sigma$. If $(A_i, \varphi_i, \epsilon_i)$ is a sequence of solutions to \eqref{eqn:vortex1} having $D$ as the zero divisor and satisfying $\epsilon_i \to 0$, then after passing to a subsequence and changing $(A_i, \varphi_i)$ by gauge transformations,
\begin{itemize}
\item $\frac{i}{2\pi } \Lambda F_{A_i}$ converges as measures to the sum of Dirac deltas $
  \delta_{D} := \sum_k m_k \delta_{x_k}$, and
\item $\nabla_{A_i} \alpha_i \to 0$, $| \alpha_i | \to 1$, and  $F_{A_i} \to 0$ in $C^0_{\loc}$ on $\Sigma \setminus D$ \footnote{By abuse of notation, we denote by the same symbol a divisor and the underlying set of points.}.
\end{itemize}

In this paper we consider the following generalisation of \eqref{eqn:vortex1}. Fix auxiliary unitary bundles $E_1, \ldots, E_N$ over $\Sigma$ together with respective connections $B_1, \ldots$, $B_N$ and non-zero integer weights $k_1, \ldots, k_N$. Let $\epsilon > 0$ and $\tau \in \R$. The generalised equations for a connection $A$ on $L$ and a section $\varphi = (\varphi^1, \ldots, \varphi^N) \in \Gamma( \oplus_{j=1}^N E_j \otimes L^{\otimes k_j})$ are
 \begin{equation} \label{eqn:generalised}
\left\{
\begin{array}{l}
\del_{A \otimes B_j} \varphi^j = 0 \qquad \textnormal{for } j = 1, \ldots,  N, \\
\epsilon^2 i \Lambda F_A + \sum_{j=1}^N k_j | \varphi^j |^2 + \tau = 0.
\end{array} \right.
\end{equation}
Equations \eqref{eqn:generalised} fit into the more general setting of \emph{framed vortex equations} discussed in \cite{bradlow2}. As before, the moduli space of solutions has a holomorphic description \cite{bw}. 

Our main result concerns sequences of solutions to \eqref{eqn:generalised} with $\epsilon \to 0$. Suppose that either the weights $k_i$ are of mixed signs or $k_i > 0$ for all $i$ and $\tau < 0$ as otherwise there are no solutions.

\begin{thm}\label{thm:generalised}
Let $(A_i, \varphi_i, \epsilon_i)$ be a sequence of solutions to \eqref{eqn:generalised} such that  $\epsilon_i \to 0$ and the sequence of norms $\| \varphi_i \|_{L^2}$ is bounded. Then there is a finite set of points $D \subset \Sigma$ such that after passing to a subsequence and applying gauge transformations $(A_i, \varphi_i  )$  converges in $C^{\infty}_{\loc}$ on $\Sigma \setminus D$. The limit $(A,\varphi)$ satisfies \eqref{eqn:generalised} with $\epsilon = 0$ on $\Sigma \setminus D$.
\end{thm}

\begin{rem}
More generally, we can assume $\tau = \tau_i$ to depend on $i$ as long as $\tau_i$ converges. If the sequence $\lambda_i := \| \varphi_i \|_{L^2}$ is unbounded, we obtain the convergence of $(A, \lambda_i^{-1} \varphi_i, \lambda_i^{-1} \epsilon_i)$ with $\tau_i = \lambda_i^{-1} \tau$. Thus, Theorem \ref{thm:generalised} describes also the limiting behaviour of solutions to \eqref{eqn:generalised} with $\epsilon=1$ and the $L^2$ norms diverging to infinity. In other words, we provide a description of the ends of the non-compact moduli space of solutions to \eqref{eqn:generalised}. This should be compared with recent results on sequences of solutions to the Hitchin equations \cite{mazzeo, mochizuki}.
\end{rem}

When $N=1$, $k_1 = 1$, and $\tau = 1$, we recover the classical vortex equations \eqref{eqn:vortex1}. As a result, we reprove and strengthen the result of \cite{hjs}. We should point out that our method of proof is different from that of \cite{hjs, mochizuki} and will be outlined at the end of this introduction. 

\begin{thm}
\label{thm:vortex}
Let $(A_i, \varphi_i, \epsilon_i)$ be a sequence of solutions to \eqref{eqn:vortex1} such that  $\epsilon_i \to 0$. Then there is a degree $d$ effective divisor $D$ on $\Sigma$ such that after passing to a subsequence and applying gauge transformations $(A_i, \varphi_i)$ converges in $C^{\infty}_{\loc}$ on $\Sigma \setminus D$ and $\frac{i}{2\pi} \Lambda F_{A_i} \to \delta_D$ as measures. The limit $(A, \varphi)$ satisfies $F_A = 0$, $|\alpha| = 1$ and $\nabla_A \alpha = 0$ on $\Sigma \setminus D$.
\end{thm} 

\begin{rem}
In the general setting of Theorem \ref{thm:generalised}, the limiting connection $A$ is no longer necessarily flat; we will see an example of this in point $(4)$ of Theorem \ref{thm:multi} below.
\end{rem}

\subsection*{Seiberg--Witten theory} 
The main application of Theorem \ref{thm:generalised} concerns generalised Seiberg--Witten equations in dimension three. To set the stage, let $Y$ be a closed Riemannian spin three-manifold. Let  $S$ be the spinor bundle and $E$, $L$ vector bundles over $Y$ with structure groups  $\SU(n)$ and $\U(1)$ respectively; we equip $E$ with a connection $B$. The \emph{Seiberg--Witten equations with multiple spinors} for a connection $A$ on $L \to Y$ and $\Psi \in \Gamma(\Hom(E, S \otimes L))$ are
\begin{equation} \left\{
\begin{array}{l}
\slashed{D}_{A \otimes B} \Psi = 0, \\
F_A = \Psi \Psi^* - \frac{1}{2} | \Psi |^2.
\end{array}
\right.
\end{equation} 
Here $\slashed{D}_{A \otimes B}$ is the Dirac operator twisted by $A$ and $B$ and in the second equation we use the identification $i \Lambda^2 T^*Y \cong i \mathfrak{su}(S)$ given by the Clifford multiplication. 

An analogous set of equations on four-manifolds was introduced in \cite{bw}. 
The three-dimensional version was studied by Haydys and Walpuski in relation to enumerative theories for associative submanifolds and $G_2$--instantons on $G_2$--manifolds \cite{donaldson-segal, walpuski, haydys4, hw}. The principal result of \cite{hw} concerns the limiting behaviour of sequences  of solutions $(A_i, \Psi_i)$ such that $\| \Psi_i \|_{L^2} \to \infty$. Haydys and Walpuski showed that there is a closed nowhere dense subset $Z \subset Y$ such that after passing to a subsequence and applying gauge transformations 
\[
   A_i \to A \textnormal{ weakly in } W^{1,2}_{\loc}
    \quad \textnormal{and} \quad 
  \Psi_i / \| \Psi_i \|_{L^2} \to \Psi \textnormal{ weakly in } W^{2,2}_{\loc}
\]
on $Y \setminus Z$, and the limiting configuration $(A, \Psi)$ defined on $Y \setminus Z$ satisfies
\begin{equation} \left\{
\begin{array}{l}
\slashed{D}_{A\otimes B} \Psi = 0, \\
 \Psi \Psi^* - \frac{1}{2} | \Psi |^2 = 0.
\end{array}
\right.
\end{equation}
Moreover, $Z$ is the zero locus of $\Psi$ in the sense of \cite{taubes} and,  if $\rank E = 2$, $A$ is flat with holonomy contained in $\Z_2$. If $\rank E > 2$, then $A$ induces a flat $\Z_2$--connection on a rank two subbundle of $E$ twisted by a line bundle; see \cite[Appendix A]{hw}.

A number of problems in this theory remain open despite their importance for the possible applications of generalised Seiberg--Witten equations to $G_2$--gauge theory:

\begin{enumerate}
  \item Taubes has made significant progress in the study of the local structure of $Z$, proving in particular that $Z$ has Hausdorff dimension at most two \cite{taubes}; yet the question whether $Z$ is a smooth curve is still unanswered.
  \item For the applications in enumerative theories it is crucial to improve the convergence statement for $(A_i,\Psi_i/\| \Psi_i\|_{L^2})$, as exemplified by \cite{doan-walpuski} where,  as part of the main proof,  $C^{\infty}$ convergence is established under the assumption that $Z$ is empty.
  \item There are two ways of associating weights to the connected components of $Z$: one based on Taubes' frequency function \cite{taubes} and one developed by Haydys using topological methods  \cite{haydys}. It is currently unknown whether these constructions are related.
  \item Haydys conjectured that, equipped with appropriate weights, $Z$ has the structure of an integral rectifiable current and that $\frac{i}{2\pi} F_{A_i}$ converges to $Z$ as currents \cite{haydys}.
\end{enumerate}

Using Theorem \ref{thm:generalised} and the methods involved in its proof, we refine the compactness theorem of \cite{hw} and solve all of the above problems in the case $Y = S^1 \times \Sigma$. 

\begin{thm}\label{thm:nsw}
Suppose that  $Y = S^1 \times \Sigma$ equipped with a product metric, $B$ is pulled back from $\Sigma$, and $(A_i, \Psi_i)$, $(A,\Psi)$, and $Z$ are as above, see also \cite[Theorem 1.5]{hw}. Then
\begin{enumerate}
  \item The singular set $Z$ is of the form $S^1 \times D$ for a degree $2d$ divisor $D = \sum_k m_k x_k$.
  \item After passing to a subsequence and applying gauge transformations 
  \[ 
  A_i \to A \quad \textnormal{and} \quad\Psi_i / \| \Psi_i \|_{L^2} \to \Psi
  \]
   in $C^{\infty}_{\loc}$ on $Y \setminus Z$.
   \item $| \Psi |^4$ extends to a smooth function on $Y$ whose zero set is $Z$ and for all $k$
   \[ 
     | \Psi (x) | = O\left( \mathrm{dist}(x,S^1 \times \{ x_k \})^{|m_k|/2} \right).
  \]
  In particular, the weight of the connected component $S^1 \times \{ x_k \}$ of $Z$ in the sense of \cite{haydys} is smaller than or equal to its weight in the sense of \cite{taubes}.
  \item If $\rank E = 2$, then $\frac{i}{2\pi} F_{A_i} \to \frac{1}{2} Z$ as currents. If $\rank E > 2$, then there is a rank two subbundle $F \subset \restr{E}{Y \setminus Z}$ such that $\Psi \in \Gamma(Y \setminus Z, \Hom(F, S \otimes L))$ and the previous statement holds if we replace $A$ and $A_i$ by the tensor product connections on $L \otimes ( \det F)^{1/2}$. Here $F$ and $\det F$ are  equipped with the unitary connections induced from $B$.  
\end{enumerate}
\end{thm}

The relationship between Seiberg--Witten monopoles with multiple spinors and generalised vortices \eqref{eqn:generalised} is the subject of the author's paper \cite{doan} where further consequences of the results presented here are explored. In particular, Theorem \ref{thm:nsw} is used to construct a compactification of the moduli space of Seiberg--Witten monopoles with multiple spinors on $S^1 \times \Sigma$ and to compare it with a corresponding algebro-geometric moduli space \cite[Theorem 1.5]{doan}. This, in turn, leads to the first known examples of the non-compactness phenomenon predicted by the Haydys and Walpuski's theorem \cite[section 8]{doan}. 

\subsection*{Symplectic vortex equations} 
Coming back to dimension two, equations \eqref{eqn:generalised} fit into the general framework of \emph{symplectic vortex equations} or \emph{gauged $\sigma$-models}. One associates vortex-type equations on $\Sigma$ to any pair $(G,M)$, where $G$ is a compact Lie group acting in a Hamiltonian way on a symplectic manifold $M$. In the spirit of Gromov--Witten theory one wishes to extract numerical invariants of $(G,M)$ from the moduli space of solutions. The parameter $\epsilon$ can be incorporated into the equations in the same way as before. In the adiabatic limit $\epsilon \to 0$, we obtain the equation for pseudoholomorphic curves in the symplectic quotient $M \sslash G$. Thus, we expect a relation between the invariants of $(G,M)$ and the Gromov--Witten invariants of $M \sslash G$. This programme has been proposed and successfully carried out by Cieliebak, Gaio, Mundet i Riera, Salamon  \cite{cgms}, \cite{gs}, and others. In order to establish Gromov compactness for symplectic vortices, more constraints are imposed on the pair $(G,M)$, a crucial condition being the properness of the Hamiltonian moment map.  Already for linear actions this is a rather restrictive assumption. The simplest example is the one discussed here with the corresponding equations \eqref{eqn:generalised}.  In this case $M = \C^n$ and $G = \U(1)$ acts diagonally with weights $(k_1, \ldots, k_N)$. The moment map is not proper unless all the weights have the same sign. Theorem \ref{thm:generalised} shows that if the properness condition is dropped we have to take under account formation of singularities in considerations regarding compactness and adiabatic limits. 

\subsection*{Outline of the proof} 
One consequence of the improperness of the moment map is that, unlike classical vortices, solutions to \eqref{eqn:generalised} do not obey an \emph{a priori} $L^{\infty}$ bound. This causes a major difficulty in establishing the convergence. While the proof in \cite{hjs} is based on local $\epsilon$--regularity estimates, we employ here a complex-geometric description of the moduli space of solutions to \eqref{eqn:generalised}. To be more specific, we use the action of $\mathcal{G}^c = C^{\infty}(\Sigma, \C^{\times})$, the group of complex automorphisms of $L$, on the space of pairs $(A, \varphi)$. The moduli space of solutions to \eqref{eqn:generalised} is homeomorphic to the quotient of the set solutions to the Cauchy--Riemann equation by $\mathcal{G}^c$; this is a Hitchin--Kobayashi type correspondence proved in  \cite{bw}. Using elliptic estimates for Dolbeault operators, we show that this quotient is compact modulo the rescaling action of $\C^{\times}$. Thus, there are complex gauge transformations $g_i = e^{f_i} u_i$ for $f_i \in C^{\infty}(\Sigma, \R)$ and $u_i \in C^{\infty}(\Sigma, \U(1))$ such that  after rescaling and applying $g_i$ the original sequence $(A_i, \varphi_i)$ converges. In order to obtain the convergence in the real rather than the complex moduli space, we need to control the functions $f_i$. The original equations \eqref{eqn:generalised} translate in this setting to a partial differential equation for $f_i$ of the form
\begin{equation}
\label{eq_generalizedkw}
\epsilon^2 \Delta f + \sum_{j=1}^n A_j e^{\alpha_j f} - \sum_{j=1}^m B_j e^{-\beta_j f} + w = 0
\end{equation}
for some functions $A_j \geq 0$, $B_j \geq 0$, $w$, and positive constants $\alpha_j$, $\beta_j$. This is a generalisation of the Kazdan--Warner equation \cite{kw}, \cite{bw}. In section \ref{sec:apriori} we establish \emph{a priori} bounds for solutions of this equation. Importantly, they are independent of $\epsilon \in (0, 1]$ and uniform on compact subsets of $\Sigma \setminus D$, where  $D$ is the set of common zeroes of  $A_j$ and $B_j$. Consequently, the Arzel\`a--Ascoli theorem guarantees the existence of a subsequence of $f_i$ converging smoothly on compact subsets of $\Sigma \setminus D$.

To the best of our knowledge, the strategy of passing to the holomorphic moduli space by means of a Hitchin--Kobayashi correspondence, obtaining good control there using $\del$--methods, and deducing from it compactness for the original sequence in the real moduli space, has not been used before. We believe that this idea---and some of the related analytical results such as Lemma \ref{lem:c0}---might be useful in studying other gauge-theoretic equations on K\"ahler manifolds.

\subsection*{Acknowledgements} The work presented in this article is part of my doctoral thesis at Stony Brook University. I am grateful to my advisor Simon Donaldson for his guidance and support. Thanks to Andriy Haydys and Thomas Walpuski for their encouragement and many helpful discussions, and to Gon\c{c}alo Oliveira, Oscar Garcia--Prada, Song Sun, Alex Waldron, and the anonymous referee for valuable comments on the previous versions of this paper. 
I am supported by the \href{https://sites.duke.edu/scshgap/}{\emph{Simons Collaboration on Special Holonomy in Geometry, Analysis, and Physics}}.

\section{Background and notation}\label{sec:background}
The set of solutions to \eqref{eqn:generalised} is invariant under the action of the gauge group $\mathcal{G}$ of unitary automorphisms of $L$, identified with $C^{\infty}(\Sigma,\U(1))$. The action of a map $u \colon \Sigma \to \U(1)$  on $(A, \varphi^1, \ldots, \varphi^N)$ is given by
\[ u (A, \varphi^1, \ldots, \varphi^N) = (A - u^{-1} du, u^{k_1} \varphi^1, \ldots, u^{k_N} \varphi^N). \]
The Dolbeault equation in \eqref{eqn:generalised}, as well as the algebraic condition $\varphi^1 \varphi^2 = 0$ in the equation \eqref{eqn:multi} below, are also invariant under the action of the \emph{complex gauge group} $\mathcal{G}^c$. It consists of complex automorphisms of $L$ and is identified with $C^{\infty}(\Sigma, \C^{\times})$, where $\C^{\times} = \C \setminus \{ 0 \}$. The action of $g \colon \Sigma \to \C^{\times}$ is given by
\[ g(A, \varphi^1, \ldots, \varphi^N) = \left(A + \overline{g}^{-1} \partial \overline{g} - g^{-1} \del g, g^{k_1} \varphi^1, \ldots , g^{k_N} \varphi^N \right).    \] 
In terms of the associated Dolbeault operators, for $s \in \Gamma(\Sigma, E_j \otimes L^{\otimes k_j})$ we have
\[ \del_{B_j, g(A)} s = g^{k_j}  \del_{B_j A} \left( g^{-k_j} s \right). \]
The action of $\mathcal{G}^c$ does not preserve the last equation in \eqref{eqn:generalised} involving the curvature. Indeed, if we write $g = e^f u$ for functions $f \colon \Sigma \to \R$ and $u \colon \Sigma \to \U(1)$, then
\[ F_{g(A)} = F_A + 2 \del \partial f, \]
or equivalently
\[ i \Lambda F_{g(A)} = i \Lambda F_A + \Delta f, \]
where $\Delta$ is the Hodge Laplacian acting on functions. 

We will need also the following lemma whose elementary proof we omit.

\begin{lem}\label{lem:flatness}
Let $L \to \Sigma$ be a Hermitian line bundle, $D \subset \Sigma$ a finite set of points, $A$ a unitary connection on $\restr{L}{\Sigma \setminus D}$ and $\alpha \in \Gamma(\Sigma\setminus D, L)$. If $\del_A \alpha = 0$, and $| \alpha | = 1$ everywhere on $\Sigma \setminus D$, then 
\[ \nabla_A \alpha = 0 \qquad \textnormal{and} \qquad F_A = 0. \]
Moreover, let $B$ be a small ball around a point $p \in D$ such that in a local unitary trivialisation $A = d + a$ for a one-form $a \in \Omega^1(B \setminus \{ p \}, i\R)$, and $\alpha$ is identified with a smooth function $\alpha \colon B \setminus \{ p \} \to S^1$. Then 
\[ \frac{i}{2\pi} \int_{\partial B}  a = \mathrm{deg} \left( \restr{\alpha}{\partial B} \right). \]
\end{lem}

%\begin{proof}
%We have
%\[ \langle \partial_A \alpha, \alpha \rangle = \langle \partial_A \alpha, \alpha \rangle + \langle \alpha, \del_A \alpha \rangle = \partial | \alpha |^2 = 0, \]
%where the angle bracket denotes the Hermitian inner product on the line bundle factor. Since $\alpha$ is nowhere vanishing on $\Sigma \setminus D$ and so $\Omega^{1,0} \otimes L$ is locally spanned by $\alpha \otimes d \bar{z}$, it follows that $\partial_A  \alpha = 0$ and as a consequence 
%\[ \nabla_A \alpha = \partial_A \alpha + \del_A \alpha = 0. \]
%In particular, $A$ on $\restr{L}{\Sigma \setminus D}$ is flat, since it admits a non-zero covariantly constant section. 
%
%As regards the second statement, assume that $B = \{ z \in \C \ | \ |z| \leq 1 \}$, $p = 0$, and let $A = d + a$ and $\alpha \colon B \setminus \{ 0 \} \to S^1$ as above. They satisfy
%\[  d\alpha + a \alpha = \nabla_A \alpha = 0, \]
%or equivalently $a = - d \log \alpha$. If we write $\alpha = e^{2\pi i \theta}$ for a multi-valued angle function $\theta$ on $B \setminus \{ 0 \}$, then
%\[ \frac{i}{2\pi} \int_{\partial B} a = - \frac{i}{2\pi} \int_{\partial B} d \log \left( e^{2\pi i \theta} \right) = \int_{\partial B} d \theta = \mathrm{deg} \left( \restr{\alpha}{\partial B} \right). \]
%\end{proof}

We end this section by restating Theorem \ref{thm:nsw} in terms of generalised vortex equations. In \cite[Theorem 1.8]{doan} we show that all irreducible solutions to the Seiberg--Witten equations with multiple spinors are pulled back from solutions to \eqref{eqn:generalised} of the following form. Using the notation of the introduction, set $N=2$, $k_1 = 1$, $k_2 = -1$. Fix a spin structure on $\Sigma$ thought of as a square root $K^{1/2}$ of the canonical bundle, and let $E$ be an $\SU(n)$-bundle; then set $E_1 = E \otimes K^{1/2}$ and $E_2 = E^* \otimes K^{1/2}$.

The equations we consider next involve a connection $A$ on $L$ and sections $\varphi^1 \in \Gamma(E \otimes K^{1/2} \otimes L)$ and $\varphi^2 \in \Gamma(E^* \otimes K^{1/2} \otimes L^*)$:
\begin{equation} \label{eqn:multi}
\left\{ 
\begin{array}{l}
\del_{A\otimes B} \varphi^1 = 0, \\
\del_{A \otimes B} \varphi^2 = 0, \\
\varphi^1 \varphi^2 = 0, \\
\epsilon^2 i \Lambda F_A + |\varphi^1|^2 - |\varphi^2|^2 = 0,
\end{array}
\right.
\end{equation}
The third equation is an additional algebraic condition for the section $\varphi^1 \varphi^2 \in \Gamma(K)$ defined as the image of $(\varphi^1,\varphi^2)$ under the pairing
\[ \Gamma(E \otimes K^{1/2} \otimes L) \times \Gamma(E^* \otimes K^{1/2} \otimes L^*) \longrightarrow \Gamma(K). \]

We will deduce Theorem \ref{thm:nsw} from the following result.

\begin{thm} \label{thm:multi}
Let $(A_i, \varphi_i, \epsilon_i)$ be a sequence of solutions to \eqref{eqn:multi} with $\| \varphi_i \|_{L^2} = 1$ and $\epsilon_i \to 0$. Then
\begin{enumerate}
 \item There exist a degree $2d$ divisor $D = \sum_k m_k x_k$ and a configuration $(A,\varphi)$ defined on $\Sigma \setminus D$ and satisfying \eqref{eqn:multi} with $\epsilon = 0$,
  \item $(A_i, \varphi_i) \to (A,\varphi)$ in $C^{\infty}_{\loc}$ on $ \Sigma \setminus D$,
\item The function $| \varphi |^4$ extends to a smooth function on all of $\Sigma$ whose zero set consists of the points in $D$ and for all $k$
   \[ 
     | \varphi (x) | = O\left( \mathrm{dist}(x,x_k)^{|m_k|/2} \right).
  \]
\item If $\rank E = 2$, then the limiting connection $A$ is flat, has holonomy contained in $\Z_2$, and $\frac{i}{2\pi} \Lambda F_{A_i} \to \frac{1}{2} \delta_D$ as measures. If $\rank E > 2$, then there exists a rank two subbundle $F \subset \restr{E}{\Sigma \setminus D}$  such that 
\[ \varphi^1 \in \Gamma(\Sigma \setminus D, F \otimes L \otimes K^{1/2}), \qquad \varphi^2 \in \Gamma(\Sigma \setminus D , F^* \otimes L^* \otimes K^{1/2}), \]
 and the previous statement holds if we replace $A$ and $A_i$ by the tensor product connections on $L \otimes (\det F)^{1/2}$. Here, $F$ and $\det F$ are  equipped with the unitary connections induced from $B$.
\end{enumerate}
\end{thm}

\begin{rem}
In contrast to Theorem \ref{thm:vortex} here the divisor $D$ need not be effective. In a way, replacing classical vortices by solutions to \eqref{eqn:multi} is analogous to replacing holomorphic sections by meromorphic sections. This idea will play a role in the proof of Theorem \ref{thm:multi}.
\end{rem}

\section{A priori estimates}\label{sec:apriori}

The main analytical input are \emph{a priori} estimates for solutions to \eqref{eq_generalizedkw}. 
%The proof uses integration by parts and other elementary techniques, but two features of the problem cause some complications. First, we need to deal with manifolds with boundary, which forces us to introduce auxiliary cut-off functions. Second, equation \eqref{eq_generalizedkw} becomes degenerate at $\epsilon = 0$. On the other hand, we need estimates which are independent of $\epsilon$, as long as it stays bounded.

\begin{prop} \label{prop:bounds}
Let $X$ be a compact Riemannian manifold with (possibly empty) boundary $\partial X$, and $\Omega \subset X$ an open subset such that $\overline{\Omega} \subset X \setminus \partial X$. Let $\epsilon_0$, $\alpha_1, \ldots , \alpha_n$, $\beta_1, \ldots , \beta_m$ be positive numbers, and let $A_1, \ldots, A_n$, $B_1, \ldots, B_m$, and $w$ be smooth functions on $X$ such that $A_j \geq 0$ and $B_j \geq 0$ for all $j$ and
\[ A_1 + \dots + A_n > 0, \qquad  B_1 + \dots + B_m > 0. \]
Then there exist constants $M_0, M_1, M_2, \ldots $, depending only on the data listed above, such that for any $\epsilon \in [0, \epsilon_0]$ and $f \in C^{\infty}(X)$ satisfying the equation
\begin{equation} \label{eqn:kw2}
\epsilon \Delta f + \sum_{j=1}^n A_j e^{\alpha_j f} - \sum_{j=1}^m B_j e^{-\beta_j f} + w = 0,
\end{equation}
the following inequalities hold:
\[ \| f \|_{C^k(\Omega)} \leq M_k \qquad \textnormal{for} \ k=0, 1, 2, \ldots \]
\end{prop}

\begin{rem} \label{rem:bounds} 
$M_k$ depends on $A_j$, $B_j$, and $w$ and their derivatives. Later we will consider sequences $\epsilon_i \to 0$ and $f_i$, $A_1^i, \ldots, A_n^i$, $B_1^i, \ldots, B_m^i$, $w_i$ satisfying for all $i$
\[ \epsilon_i \Delta f_i + \sum_{j=1}^n A_j^i e^{\alpha_j f_i} - \sum_{j=1}^m B_j^i e^{-\beta_j f_i} + w_i = 0. \]
It will be clear from the proof that in this case the $C^k$ estimate still holds for large $i$ (depending on $k$) provided that $A_j^i$, $B_j^i$, and $w_i$ converge smoothly to $A_j$, $B_j$, and $w$ respectively, satisfying  
\[ A_1 + \dots + A_n > 0, \qquad  B_1 + \dots + B_m > 0. \]
\end{rem}

%\begin{rem}
%We expect that the statement is still true when $B_1 = \ldots = B_m = 0$ and $w  > 0$. Although the proof presented here does not immediately generalise to this case, it does establish $C^k$ estimates for $k \geq 1$ under the additional assumption that $\| f \|_{C^0(X)} \leq K$ for some constant $K$. Then $M_k$ depends also on $K$. We will use this in the proof of Theorem \ref{thm:vortex}, in which the $C^0$ bound is clear for other reasons.
%\end{rem}

The proof of Proposition \ref{prop:bounds} is preceded by three lemmas.

 \begin{lem}\label{lem:cutoff}
Let $X$ and $\Omega$ be as in Proposition \ref{prop:bounds}, and let $V_0 \subset X$ be an open subset such that $\overline{\Omega} \subset V_0$. Then there exist an open subset $V$ and $\phi \in C^{\infty}(X)$ such that:
\begin{enumerate}
\item $\Omega \subset V \subset V_0$. 
\item $0 < \phi \leq 1$ on $V$.
\item $\phi = 1$ on $\Omega$. 
\item $\phi = 0$ on $X \setminus V$.
\item There is a constant $K$ such that for any $\alpha \in [0,2)$, 
\[ \sup_{V'} \frac{| \nabla \phi |^2}{\phi^{\alpha}} \leq \frac{K}{(2-\alpha)^4}. \]
\end{enumerate}
\end{lem}

\begin{proof}
Let $V' \subset X$ be any open subset such that $\overline{\Omega} \subset V' \subset V$ and $V'$ has smooth boundary. We can construct such a subset for example by taking any smooth function $h \colon X \to \R$ with $h < 0$ on $\Omega$ and $h> 1$ on $X \setminus V$, and setting $V' = h^{-1}( (- \infty, c))$ where $c \in (0,1)$ is a regular value of $h$. Let $N = h^{-1}(c)$ be the boundary of $V'$. Assume for simplicity that $X$ is orientable. By the tubular neighbourhood theorem, there is an embedding
\[ (- \epsilon, \epsilon) \times N \hookrightarrow X \]
such that $\{ 0 \} \times N$ is mapped diffeomorphically onto $N \subset X$ and the image of $(0, \epsilon) \times N$ is contained in $V'$.  We may also assume that the image of this embedding is disjoint from $\overline{\Omega}$.  

Using a partition of unity (and passing to a slightly smaller $\epsilon$), we construct a function $\phi$ with properties $(1)$--$(4)$, which for $(t,x) \in (-\epsilon, \epsilon) \times N$ agrees with
\[ \phi(t, x) = \left\{
\begin{array}{ll}
M \exp\left( - \frac{1}{t} \right) & \textnormal{for } t \in (0, \epsilon) \\
0 & \textnormal{for } t \in (- \epsilon, 0) \\
\end{array}
\right. \]
for some constant $M$ required for the normalisation $\| \phi \|_{C^0(X)} = 1$. Now let $\alpha < 2$. Away from $N$ we have $\phi > 0$ and $| \nabla \phi |^2 / \phi^{\alpha}$ is bounded. In a neighbourhood of $N$, 
\[ \frac{ | \nabla \phi |^2}{\phi^{\alpha} } \leq \frac{C}{t^4} \exp\left( -\frac{(2-\alpha)}{t} \right) \]
for some constant $C$ depending on the Riemannian metric on $X$ and the embedding of the tubular neighbourhood. Define
\[ 
  g(t) = \frac{1}{t^4} \exp\left( -\frac{(2-\alpha)}{t} \right).
\]
Then $g$ is smooth and bounded on $[0, \infty)$ and its global maximum is
\[ g\left( \frac{2-\alpha}{2} \right) = \frac{ 4e^{-2} }{(2-\alpha)^4}, \]
which shows that $| \nabla \phi |^2 / \phi^{\alpha} \leq K (2-\alpha)^{-4}$ as desired.
\end{proof}

\begin{lem} \label{lem:c0}
Let $X$ and $\Omega$ be as in Proposition \ref{prop:bounds}. Fix positive numbers $\epsilon_0$, $p$, and $\gamma > 1$, and consider functions $A \in C^{\infty}(X)$ and $Q \in C^{\infty}(X \times (0, \infty))$ such that  for all $(x,y) \in X \times (0, \infty)$,
\[ A(x) \geq \eta > 0 \qquad \textnormal{and} \qquad | Q(x,y) | \leq \sum_{j=1}^k a_j(x) y^{\gamma_j}, \]
where $a_j \in C^{\infty}(X)$ and $\gamma_j < \gamma$. Under these assumptions there exists a constant $M$ such that for any $\epsilon \in [0, \epsilon_0]$ and  $u \in C^{\infty}(X)$ satisfying $u \geq 0$ and 
\begin{equation} \label{eqn:c0} \epsilon \Delta u + A u^{\gamma} + Q(x,u) \leq 0, 
\end{equation}
the following inequality holds:
\[ \| u \|_{L^p(\Omega)} \leq M. \]
Moreover, $M$ depends only on $\Omega, X, \epsilon_0, p, \gamma, \gamma_j, \eta$, and the norms $\| a_j \|_{L^q(X)}$ for a certain $q < \infty$  depending on $p$.  
\end{lem}

%\begin{rem}
%\label{rem:simpleproof}
%Before we begin the proof, observe that in the degenerate case $\epsilon = 0$ the statement is obvious as \eqref{eqn:c0} implies that
%\[ \eta u^{\gamma} \leq A u^{\gamma} \leq |Q(x,u)| \leq \sum_{j=1}^k a_j(x) u^{\gamma_j}. \]
%Set $\gamma_{\mathrm{max}} = \max\{ \gamma_1, \ldots, \gamma_k \}$. Under the assumption that $u(x) \geq 1$ we have
%\[ (u(x))^{\gamma - \gamma_{\mathrm{max}}}  \leq \eta^{-1} \sum_{j=1}^k a_j(x) (u(x))^{\gamma_j - \gamma_{\mathrm{max}}} \leq \eta^{-1} \sum_{i}^k \sup | a_j| \leq C, \]
%Since $\gamma_{\mathrm{max}} < \gamma$, this results in the upper bound $u \leq \max\{ 1, C^{\gamma_{\mathrm{max}} - \gamma} \}$. The general proof imitates this simple argument. 
%\end{rem}

\begin{proof}
We adopt here the convention that $C$ always denotes a constant depending only on the fixed data and not on $u$ or $\epsilon$. Its value might change from line to line. 

It is enough to prove the statement for \emph{some} $\epsilon_0 > 0$, which we will later assume to be sufficiently small. Indeed, the corresponding statement for any other $\epsilon_0' > \epsilon_0$ can be reduced to the one for $\epsilon_0$ by multiplying both sides of \eqref{eqn:c0} by $\epsilon_0 / \epsilon_0'$ at the cost of appropriately scaling $A$ and $Q$.  
Furthermore, it suffices to prove the statement for $p=1$ since for $p >1$ we have
\[ \begin{split} \Delta( u^p ) &= - p(p-1) u^{p-2} | \nabla u |^2 + p u^{p-1} \Delta u \\
& \leq -pA u^{p - 1 + \gamma} - pu^{p-1} Q(x,u) \\
& \leq - A' (u^p)^{\gamma'} - Q'(x,u^p), \end{split} \]
where 
\[ A' = pA, \qquad \gamma' = 1 + \frac{\gamma-1}{p}, \qquad Q'(x,y) = py^{ \frac{p-1}{p} }Q(x, y^{\frac{1}{p}}). \]
We easily check that the new data $(A', \gamma', Q')$ satisfies the hypotheses of the lemma. The only non-trivial condition is the estimate for $Q'$ which follows from
\[ | Q'(x,y) | \leq  p \sum_{i}^k a_j(x) y^{1 + \frac{\gamma_j - 1}{p}} = p \sum_{i}^k a_j(x) y^{\gamma_j'}. \]
Note that $\gamma_j' < \gamma'$. Therefore, the statement for $p > 1$ reduces to that for $p=1$ after replacing $(u, A, \gamma, Q)$ by $(u^p, A', \gamma', Q')$. 

In fact, we will bound $\| u \|_{L^{1+\gamma}(\Omega)}$ . Let $V_0 \subset X$ be an open subset containing $\overline{\Omega}$ such that the volume of $V_0 \setminus \Omega$ is sufficiently small. We will specify later what we mean by that, but for the moment let us stress that the choice of  $V_0$ will depend only on the fixed data and not the function $u$. Once $V_0$ is fixed, choose a subset $V \subset V_0$ and a bump function $\phi \in C^{\infty}(X)$ as in Lemma \ref{lem:cutoff}.
Note that $\mathrm{vol}(V \setminus \Omega) \leq \mathrm{vol}(V_0 \setminus \Omega)$.

Multiply inequality \eqref{eqn:c0} by $u \phi^2$ and integrate it over $X$: 
\begin{equation} \label{eqn:c0_1}
\int_X   \epsilon (\Delta u) u \phi^2 + A u^{1+\gamma} \phi^2 + Q u \phi^2 \leq 0 .
 \end{equation}
Since $\phi$ has compact support, integration by parts yields
\[ \begin{split} \int_X ( \Delta u) u \phi^2 &=   \int_X \langle \nabla u, \nabla (u \phi^2) \rangle \\
& =  \int_X \langle \nabla u , 2u\phi \nabla \phi + \phi^2 \nabla u \rangle \\
& = \int_X 2 u \phi \langle \nabla u, \nabla \phi \rangle + \phi^2 | \nabla u |^2 \\
&=  \int_X | u \nabla \phi + \phi \nabla u |^2 - u^2 | \nabla \phi |^2 \\
& \geq - \int_X u^2 | \nabla \phi |^2.
\end{split}
\]
Together with \eqref{eqn:c0_1}, this implies the inequality
\[
\int_X - \epsilon u^2 | \nabla \phi |^2 + A u^{1+\gamma} \phi^2  + Q u \phi^2 \leq  0.
\]
Recall that $\phi$ is supported in $V$ and $\phi = 1$ on $\Omega$. Let $P = V \setminus \Omega$ so that $V = \Omega \cup P$. Splitting the integral on the left-hand side according to this decomposition and rearranging the inequality, we obtain
\begin{equation} \label{eqn:c0_2} \int_{\Omega} A u^{1+\gamma} + I \leq \int_{\Omega} | Q | u,
\end{equation}
where we have collected all integrals over $P$ into a single term,
\[ I = I_0 + I_1 + I_2, \]
\[ I_0 = \int_P A u^{1 + \gamma} \phi^2, \quad I_1 = \int_P Q u \phi^2, \quad I_2 = -  \epsilon \int_P u^2 | \nabla \phi |^2. \]
The next goal is to estimate $I$. We will show that for a suitable choice of $V$ and $\epsilon_0$, depending only on the initial data and not on $u$, we may assume that $I \geq 0$, provided that $\epsilon \leq \epsilon_0$. Strictly speaking, this will not always be true, but in the case when our estimate fails, we will obtain an upper bound for $| I |$ so that we can move it to the right-hand side of \eqref{eqn:c0_2}.

Before proving this, let us show that the inequality $I \geq 0$ gives us a bound for $\| u \|_{L^{1+\gamma}(\Omega)}$. If $I \geq 0$, then by \eqref{eqn:c0_2} and H\"older's inequality,
\begin{equation} \label{eqn:c0_3} \begin{split} \eta \int_{\Omega} u^{1+\gamma} & \leq \int_{\Omega} A u^{1+\gamma} \leq \int_{\Omega} | Q | u  \leq \int_{\Omega} \sum_{j=1}^k a_j u^{1+\gamma_j}  \\
&\leq  \sum_{j=1}^k \| a_j \|_{L^{p_j}(X)} \left( \int_{\Omega} u^{1+\gamma} \right)^{1/q_j},
\end{split} \end{equation}
where the H\"older exponents $p_j$ and $q_j$ are given by
\[ q_j = \frac{1+\gamma}{1+\gamma_j}, \qquad \frac{1}{p_j} + \frac{1}{q_j} = 1. \]
Note that $q_j > 1$ for each $i$, because $\gamma_j < \gamma$.  An equivalent way of writing \eqref{eqn:c0_3} is
\[ \| u \|_{L^{1+\gamma}(\Omega)}^{1+\gamma} \leq \eta^{-1} \sum_{j=1}^k \| a_j \|_{L^{p_j}(X)} \| u \|_{L^{1+\gamma}(\Omega)}^{\frac{1+\gamma}{q_j}}, \]
which, in view of $q_j > 1$, results in an upper bound for $\| u \|_{L^{1+\gamma}(\Omega)}$. 
%(This is, of course, the same argument as the one used in the simplified proof in Remark \ref{rem:simpleproof}.) 
The dependance of the bound on the initial data is clear. 

In order to finish the proof, it remains to estimate the integral $I = I_0 + I_1 + I_2$.  We will deal separately with each of the three terms. The first one contributes positively to $I$ and is bounded below by
\begin{equation} \label{eqn:c0_4} I_0 \geq \eta \int_P u^{1 + \gamma} \phi^2. 
\end{equation}
The terms that can contribute negatively are $I_1$ and $I_2$. We estimate the former using our assumption on $Q$ and H\"older's inequality:
\[ \begin{split} | I_1| &\leq \sum_{j=1}^k \int_P | a_j | u^{1+\gamma_j} \phi^2 \\
& \leq \sum_{j=1}^k \left( \int_P |a_j|^{p_j} \right)^{1/{p_j}} \left( \int_P u^{1+\gamma} \phi^{2q_j} \right)^{1/q_j} \\
& \leq  \sum_{j=1}^k \mathrm{vol}(P)^{1/2p_j}  \| a_j \|_{L^{2p_j}(X)}\left( \int_P u^{1+\gamma} \phi^{2q_j} \right)^{1/q_j},
\end{split} \]
where the H\"older exponents $p_j$ and $q_j$ are as before. Let $S = \int_P u^{1+\gamma} \phi^{2q_j}$. If $S \leq 1$, then $|I_1|$ is bounded by a constant independent of $u$, say $C$, and we can move it on the right-hand side of \eqref{eqn:c0_2}. Next we replace $I$ by the sum of the remaining two terms $I' = I_0 + I_2$ and if we can show that $I' \geq 0$, then repeating the previous discussion we arrive at a bound for $\| u \|_{L^{1+\gamma}(\Omega)}$ with an extra term involving $C$. Thus, let us assume that $S \geq 1$. In this case, we have $S^{1/q_j} \leq S$ and
 \[ \begin{split} | I_1| &\leq \sum_{j=1}^k \mathrm{vol}(P)^{1/2p_j}   \| a_j \|_{L^{2p_j}(X)} \int_P u^{1+\gamma} \phi^{2q_j} \\
 & \leq \sum_{j=1}^k \mathrm{vol}(P)^{1/2p_j}  \| a_j \|_{L^{2p_j}(X)} \int_P u^{1+\gamma} \phi^2,
 \end{split} \]
 where we have also used that $\phi \leq 1$ and so $\phi^{2q_j} \leq \phi^2$. Comparing the right-hand side of the inequality with the previously obtained upper bound \eqref{eqn:c0_4} for $I_0$ we see that if $P$ has sufficiently small volume (which can be guaranteed by the choice of the initial open set $V_0$), then
  \[ |I_1| \leq \frac{I_0}{2}. \]
Furthermore, how small $P$ has to be depends only on $\eta$ and $\| a_j \|_{L^{p_{\max}}(X)}$, where $p_{\max} = \max\{ p_1, \ldots, p_k \}$. 
Note that at this point the sets $V$ and $P$ are chosen and will not be changed.

The second potentially negative term $I_2$ is dealt with in a similar manner. For every $\alpha \in \R$, H\"older's inequality implies that
\begin{equation} \label{eqn:c0_5} \begin{split} |I_2| &= \epsilon \int_P u^2 | \nabla \phi|^2 = \epsilon \int_P  \frac{ | \nabla \phi |^2}{\phi^{\alpha}}  u^2 \phi^{\alpha}\\
& \leq \epsilon  \left\| \frac{ | \nabla \phi |^2 }{\phi^{\alpha}} \right\|_{L^q(P)} \left( \int_P u^{1+\gamma} \phi^{ \frac{\alpha (1+\gamma)}{2}} \right)^{\frac{2}{1+\gamma}},
\end{split}
\end{equation}
where $q$ is given by $1/q + 2/(1+\gamma) = 1$.  Observe that $1+ \gamma > 2$, so we can choose $\alpha$ so that
\[ \frac{4}{1+\gamma} < \alpha < 2. \]
Then, $| \nabla \phi |^2 / \phi^{\alpha}$ is bounded on $P$, and the first factor on the right-hand side of \eqref{eqn:c0_5} is finite. As regards the integral in the second factor, assuming as before that it is greater than or equal to one (as otherwise we can rearrange and get a bounded factor on the right-hand side of \eqref{eqn:c0_2}), we arrive at
\[ | I_2| \leq C \epsilon \int_P u^{1+\gamma} \phi^{ \frac{\alpha (1+\gamma)}{2}} \leq C \epsilon \int_P u^{1+\gamma} \phi^2, \]
where we have used that  $\phi <1$ and $\alpha(1+\gamma)/2 > 2$. 
Note that here the constant $C$ depends on the choice of $P$ and can be potentially large.
However, we still have the freedom to choose $\epsilon$ small enough---as remarked at the beginning of the proof, it suffices to establish an estimate for $\epsilon$ sufficiently small.
Thus, comparing the right-hand side of the above inequality with the lower bound \eqref{eqn:c0_4} for $I_0$, we conclude that if $\epsilon_0$ is sufficiently small, then for all $\epsilon \leq \epsilon_0$,
\[ | I_2| \leq \frac{I_0}{2}. \]
Together with the estimate for $|I_1|$, this implies that $I = I_0 + I_1 + I_2$ is non-negative (or else we can rearrange \eqref{eqn:c0_2}), which finishes the proof of the lemma. 
\end{proof}

Our proof does not work in the case $\gamma = 1$. What fails  is the last estimate for $I_2$, because we cannot set $\alpha = 2$. Indeed, there is no cut-off function $\phi$ such that  $| \nabla \phi | / \phi$ is bounded. However, we can still prove a slightly weaker statement.

\begin{lem}
  \label{lem:c0weaker}
  If $\gamma = 1$, then the statement of Lemma \ref{lem:c0} still holds provided that $u$ satisfies an estimate $\| u \|_{L^{2p}(X)} \leq K \epsilon^{-1}$ for some constant $K$. Apart from the rest of the data, the final bound for $\| u \|_{L^p(\Omega)}$ depends also on $K$. 
\end{lem}
\begin{proof}
Suppose for simplicity that $p=1$, so that 
\[  \| u \|_{L^4(X)} \leq K\epsilon^{-1}. \]
Following the proof of Lemma \ref{lem:c0}, we can obtain a bound for $\| u \|_{L^2(\Omega)}$  The only modification that we have to make is the estimate \eqref{eqn:c0_5} which now should be
\[ \begin{split} 
| I_2 | &= \epsilon \int_P u^2 | \nabla \phi |^2 = \epsilon \int_P \left(u^{1/2} \right) \left( u^{3/2} \phi^{\alpha} \right) \frac{ | \nabla \phi |^2}{\phi^{\alpha}} \\
&\leq \epsilon  \|  u^{1/2} \|_{L^8(X)} \left\| \frac{ | \nabla \phi |^2}{\phi^{\alpha}} \right\|_{L^8(P)} \left( \int_p u^2 \phi^{\frac{4\alpha}{3}} \right)^{3/4} \\
&\leq \epsilon^{1/2} K^{1/2} \left\| \frac{ | \nabla \phi |^2}{\phi^{\alpha}} \right\|_{L^8(P)} \left( \int_P u^2 \phi^{\frac{4\alpha}{3}} \right)^{3/4},
\end{split} \]
where we have used H\"older's inequality with weights $(8,8,4/3)$. Now for $\alpha$ satisfying $2 > \alpha > 3/2$, the function $| \nabla \phi |^2 / \phi^{\alpha}$ is bounded $4\alpha / 3 > 2$, so that (again, assuming that the integral on the right-hand side is greater than one) we obtain
\[ | I_2 | \leq C \epsilon^{1/2} \int_P u^2 \phi^2. \]
Recall that in the case $\gamma = 1$, the positive integral $I_0$ is bounded below by
\[  I_0 \geq \eta \int_P u^2 \phi^2, \]
so that for $\epsilon$ small enough we have $ I_2 \leq I_0 / 2$. This leads to a bound for $\| u \|_{L^2(\Omega)}$ as before. In the same way we obtain a bound for $\| u \|_{L^p (\Omega)}$ 
\[ 
  \| u \|_{L^{2p}(\Omega)} \leq K\epsilon^{-1}. 
\]
\end{proof}

\begin{proof}[Proof of Proposition \ref{prop:bounds}]
We will establish bounds of the form $\| f \|_{W^{k,p}(\Omega)} \leq M_{k,p}$  for all $k$ and $p$ by induction over $k$. Let us start with $k=0$. Let $\eta > 0$ be such that 
\[ A_1 + \cdots + A_n \geq n \eta, \qquad \textnormal{and} \qquad B_1 + \cdots + B_m \geq m \eta, \]
and set $\Omega_j = \Omega \cap \{ A_j \geq \eta \}$. The subsets $\{ A_j \geq \eta \}_{j=1,\ldots, n}$ cover $X$ and therefore $\Omega_1, \ldots, \Omega_n$ cover $\Omega$. Let $u = e^f$. For any given $j$ we have
\[ \begin{split} \epsilon \Delta u &= - \epsilon e^f | \nabla f|^2 + \epsilon e^f \Delta f  \\
&\leq u \left( - \sum_{j=1}^n A_j u^{\alpha_j} + \sum_{j=1}^n B_j u^{- \beta_j} - w \right) \\
&\leq -A_j u^{1 + \alpha_j} + \sum_{j=1}^n B_j u^{1-\beta_j} - w u, 
\end{split} \]
or equivalently,
\[ \epsilon \Delta u + A_j u^{1+ \alpha_j} + Q(x,u) \leq 0, \]
where $Q(x,u) = - \sum_j B_j u^{1-\beta_j} + wu$. It follows from Lemma \ref{lem:c0} that $\| u \|_{L^p(\Omega_j)}$ is bounded by a constant depending only on the fixed data. Since $\Omega_1, \ldots, \Omega_n$ cover $\Omega$, we obtain a bound for $\| u \|_{L^p(\Omega)}$. Similarly, considering the subsets $\Omega \cap \{ B_j \geq \eta \}$ and the function $e^{-f}$ we find bounds for $\| e^{-f} \|_{L^p(\Omega)}$. Combining them with the inequality
\[ | f| \leq e^f + e^{-f}, \]
we obtain bounds for $\| f \|_{L^p(\Omega)}$.

Suppose that $W^{k-1,p}$ bounds have been established for some $k \geq 1$ and all $p$. We may assume that they hold on a slightly larger domain containing $\overline{\Omega}$, which we assume to be all of $X$ to keep the notation simple. First consider the case when $k = 2l$ is even. Consider the function 
\[ v = \Delta^l f =  \underbrace{\Delta \cdots \Delta}_{l \textnormal{ times}} f. \]
Applying $\Delta^l$ to both sides of \eqref{eqn:kw2} and using the formula
\[ \Delta( gh ) = g \Delta h - 2 \langle \nabla g, \nabla h \rangle + h\Delta g, \]
we inductively show that $v$ satisfies a differential equation of the form
\begin{equation} \label{eqn:bounds}
 \epsilon \Delta v + A v + P(e^{\alpha_j f}, e^{- \beta_j f}, \nabla f, \ldots , \nabla^{2l-1} f ) = 0, 
 \end{equation}
where
\[ A = \sum_{j=1}^n \alpha_j A_j e^{\alpha_j f} + \sum_{i=1}^m \beta_j B_j e^{- \beta_j f}, \]
and $P$ is a polynomial function of the functions $e^{\alpha_j f}$, $e^{-\beta_j f}$, and the first $2l-1$ covariant derivatives of $f$. Its coefficients depend only on $A_j$, $B_j$, $w$, and their derivatives. In particular, $P$ is a finite sum $P = \sum_{\gamma} P_{\gamma}$ say, where each term $P_{\gamma}$ satisfies an inequality of the form
\[ | P_{\gamma} | \leq C e^{a f} e^{- b f} | \nabla f |^{c_1} \ldots | \nabla^{2l-1} f |^{c_{2l-1}} \]
with some exponents $a,b,  c_1, \ldots, c_{2l-1}$ and a coefficient $C$ depending only on $A_j$, $B_j$, $w$, and their derivatives. Therefore, by the induction hypothesis and H\"older's inequality, we can bound the $L^p$ norm of $P$ for any $p$, by a constant depending only on the fixed data and not on $f$. 

At every point $A$ is bounded below by either $\sum_j \alpha_j A_j$ or $\sum_j \beta_j B_j$,  depending on the sign of $f$. In any case, there is a positive  constant $\tilde{\eta}$, depending only on the fixed data, such that $A \geq \tilde{\eta}$. Note that $v$ is raised to the first power in \eqref{eqn:bounds}. Thus, we are in place to apply Lemma \ref{lem:c0weaker} to obtain a bound for $\| v \|_{L^p(\Omega)}$. In order to do so, we need to make sure that $v$ obeys an estimate of the form 
\[ 
  \| v \|_{L^{2p}(\Omega)} \leq K_p \epsilon^{-1}
\] 
for some constant $K_p$. Such an estimate follows from the induction hypothesis and the fact that $v' =  \Delta^{l-1} f$ satisfies an equation analogous to \eqref{eqn:bounds}:
\[ \epsilon v = \epsilon \Delta v' = - Av' + P' (e^{\alpha_j f}, e^{-\beta_j f}, \nabla f, \ldots, \nabla^{2l-3} f ), \]
where $P'$ is a polynomial function as before. Since the right-hand side depends only on the derivatives of $f$ up to the order $2l-2$, we obtain an estimate for $\epsilon v$ as required. Thus, Lemma \ref{lem:c0weaker} yields a bound for $\| v \|_{L^p(\Omega)} $. Of course, we can as well assume that it holds on a slightly larger domain containing $\overline{\Omega}$. Then, in view of $v = \Delta^l f$, the elliptic estimate for the Laplacian implies a bound for $\| f \|_{W^{2l, p}(\Omega)}$. This finishes the proof of the induction step in the case $k = 2l$.

The odd case $k = 2l+1$ is similar. Assume that the assertion is true for $k-1 = 2l$. Let $v = \Delta^l f$ as before and $\psi = | v |^2$. By the Bochner formula,
\[ \begin{split} \frac{1}{2} \Delta \psi &= - | \nabla^2 \psi |^2 - \mathrm{Ric}( \nabla \psi, \nabla \psi ) + \langle \nabla \psi, \nabla ( \Delta \psi ) \rangle \\
& \leq  \| \mathrm{Ric} \|_{C^0(X)} | \nabla \psi |^2 +  \langle \nabla \psi, \nabla ( \Delta \psi ) \rangle, \end{split} \]
where $\mathrm{Ric}$ is the Ricci curvature of $X$. After taking the gradient of \eqref{eqn:bounds} and plugging it to the inequality above, we arrive at
\[ \frac{\epsilon}{2} \Delta \psi + \left(A - \epsilon \| \mathrm{Ric} \|_{C^0(X)} \right) \psi + Q(e^{\alpha_j f}, e^{-\beta_j f}, \nabla f, \ldots, \nabla^{2l} f ) \psi^{1/2} \leq 0 , \]
where $Q$ is a polynomial function of $e^{\alpha_j}$, $e^{- \beta_j f}$, and the first $2l$ derivatives of $f$. Provided that $\epsilon$ is sufficiently small, the function $A - \epsilon \| \mathrm{Ric} \|$ is bounded below by a positive constant and we can apply Lemma \ref{lem:c0weaker} as before to obtain a bound for $\| \psi \|_{L^p(\Omega)}$. Again, by the elliptic estimate for the Laplacian, this yields  $W^{k,p}$ bounds for $f$. The statement for general $\epsilon \leq \epsilon_0$ follows from a scaling argument as described in the proof of Lemma \ref{lem:c0}. 
\end{proof}

\section{Proofs of the theorems}\label{sec:proofs}
We prove the theorems in the order of increasing generality.
%We begin with the case of classical vortices described by Theorem \ref{thm:vortex}. We will then prove Theorems \ref{thm:multi} and \ref{thm:nsw} dealing with the dimensional reduction of the Seiberg--Witten equations with multiple spinors. Finally, we discuss how these proofs can be adapted to the general setting of Theorem \ref{thm:generalised}. 

\begin{proof}[Proof of Theorem \ref{thm:vortex}]
Since $\epsilon_i \to 0$, we may assume that none of the sections $\varphi_i$ is identically zero. Let $\mathcal{A}$ be the space of unitary connections on $L$ and  $\mathcal{G}^c$ be the complex gauge group of $L$, that is the space of smooth maps from $\Sigma$ to $\C^{\times}$. 

\begin{step}[Convergence modulo $\mathcal{G}^c$] We claim that there are sequences of complex gauge transformations $g_i \in \mathcal{G}^c$  such that, after passing to a subsequence, $g_i(A_i, \varphi_i)$ converges in $C^{\infty}(\Sigma)$ to a pair $(A', \varphi')$. The limiting section $\varphi'$ is not identically zero and satisfies $\del_{A'} \varphi' = 0$.  

Let us prove this claim. The quotient $\mathcal{A} / \mathcal{G}^c$  with the $C^{\infty}$ topology is homeomorphic to the Jacobian torus $H^1(\Sigma, \R) / H^1(\Sigma, \Z)$. In particular, it is compact and there is a sequence of $g_i \in \mathcal{G}^c$ such that, after passing to a subsequence, $A_i' = g_i A_i$ converges in $C^{\infty}$ to a connection $A'$, say.  After replacing $g_i$ by $\mu_i g_i$, where  $\mu_i = \| g_i \varphi_i \|_{L^2}^{-1}$ , we may assume that $\| g_i \varphi_i \|_{L^2} = 1$ for all $i$.  Note that the constant gauge transformations $\mu_i$ act trivially on the space of connections so that we still have $A_i' \to A'$. The final remark about our choice of $g_i$ is that we will assume them to be purely "imaginary" gauge transformations. Any complex gauge transformation is of the form $g = u e^f$ for a $\U(1)$ gauge transformation $u$ and real function $f \colon \Sigma \to \R$. By incorporating the $\U(1)$ part into the original sequence $(A_i, \varphi_i, \beta_i)$ we may assume that $g_i = e^{f_i/2}$ for a smooth function $f_i \colon \Sigma \to \R$.  

Set $\varphi_i' = g_i \varphi_i$. The action of $\mathcal{G}^c$ preserves the Cauchy-Riemann equation:
\[ \del_{A_i'} \varphi_i' = 0. \]
As a consequence,
\[ \| \del_{A'} \varphi_i' \|_{L^2} = \| (\del_{A'} - \del_{A_i'}) \varphi_i' \|_{L^2} \leq \| A' - A_i' \|_{L^{\infty}} \| \varphi_i' \|_{L^2}  \]
Since $A_i \to A'$, the sequence of norms $\| \del_{A'} \varphi_i' \|_{L^2}$ is bounded by a number independent of $i$. From the elliptic estimate for $\del_{A'}$ we conclude that the sequence $\varphi_i'$ is bounded in $W^{1,2}$. Bootstrapping gives us $C^k$ for any $k$ and we can choose a subsequence (denoted for simplicity by the same symbols) that converge in $C^{\infty}$ to a section $\varphi'$, say, satisfying
\[ \del_{A'} \varphi' = 0 \qquad \textnormal{and} \qquad \| \varphi' \|_{L^2} = 1, \]
which finishes the proof of the claim.
\end{step}

\begin{step}[$C^0$ estimates]
Let $D$ be the set of zeroes of $\varphi'$. Counted with multiplicities, there are exactly $d = \deg(L)$ of them. The next step is to show that the sequence $f_i$ is uniformly bounded on compact subsets of $\Sigma \setminus D$. First, we compute
\[ \begin{split} 
\epsilon_i^2 \left( 2 i \Lambda F_{A_i'} \right) &= \epsilon_i^2 \left( 2i \Lambda F_{A_i} + \Delta f_i \right) \\
&= 1 - | \varphi_i |^2 + \epsilon_i^2 \Delta f_i \\
&= 1 - e^{-f_i} | \varphi_i' |^2 + \epsilon_i^2 \Delta f_i,
\end{split} \]
so after rearranging, we obtain the following partial differential equation for $f_i$:
\begin{equation} \label{eqn:vortexproof1}
\begin{split}
\epsilon_i^2 \Delta f_i &= e^{-f_i} | \varphi_i' |^2 - 1 + \epsilon_i^2 \left( 2 i \Lambda F_{A_i'} \right) \\
&=   q_i e^{-f_i} - w_i,
\end{split} 
\end{equation}
where $q_i = | \varphi_i' |^2$ and $w_i = 1 - \epsilon_i^2 \left( 2 i \Lambda F_{A_i'} \right)$. Now set $u_i = e^{f_i}$. Then
\[ \begin{split}
\epsilon_i^2 \Delta u_i &=  \epsilon_i^2 \left( - e^{f_i} | \nabla f_i |^2 +  e^{f_i} \Delta f_i \right) \\
& \leq q_i - w_i u_i. 
\end{split} \]
Since $w_i \to 1$ uniformly, for $i$ large enough we have $w_i \geq 1/2$. The functions $q_i$ are bounded because $\varphi_i'$ converges. Thus, the maximum principle yields an upper bound for $u_i$, and consequently for $f_i$. On the other hand, we easily compute that 
\[  \epsilon_i^2  \Delta |\varphi_i|^2 + 2 \epsilon_i^2 | \partial_{A_i} \varphi_i |^2 = |\varphi_i|^2 \left( 1 -|\varphi_i|^2 \right), \]
which again by the maximum principle shows that $|\varphi_i|^2 \leq 1$ for all $i$. Since $|\varphi_i|^2  = e^{-f_i} | \varphi_i '|^2$ and $| \varphi_i' |^2$ converges uniformly to $| \varphi' |^2$, it follows that $f_i$ is bounded below uniformly on compact subsets of $\Sigma \setminus D$. 
\end{step}

\begin{step}[Convergence outside $D$] 
Once the $C^0$ estimate is established, it follows from equation \eqref{eqn:vortexproof1},  Proposition \ref{prop:bounds}, and Remark \ref{rem:bounds} that the sequence $f_i$ is bounded uniformly with all derivatives on compact subsets of $\Sigma \setminus D$. Thus, we can choose a subsequence of $f_i$ which converges uniformly with all derivatives on compact subsets of $\Sigma \setminus D$ to a smooth function $f \colon \Sigma \setminus D \to \R$. Let $g = e^{f/2}$ be the corresponding complex gauge transformation. Set $(A, \varphi) = ( g^{-1} A', g^{-1} \varphi')$. The pair is well-defined on $\Sigma \setminus D$ and $(A_i, \varphi_i) \to (A, \varphi)$ in $C^{\infty}_{\loc}(\Sigma \setminus D$. Indeed, we have
\[ \begin{split} 
\varphi_i - \varphi &= g_i^{-1} \varphi_i' - g^{-1} \varphi' \\
&= g_i^{-1} \varphi_i' - g_i^{-1} \varphi' + g_i^{-1} \varphi' - g^{-1} \varphi' \\
&= g_i^{-1} \left( \varphi_i' - \varphi' \right) + \left( g_i^{-1} - g^{-1} \right) \varphi', 
\end{split}, \]
so the convergence of $\varphi_i'$ and $g_i'$ guarantees that for any compact subset $K \subset \Sigma \setminus D$ there are constants $M_{l,K}$ for $l = 0,1, \ldots$ such that
\[ \| \varphi_i - \varphi \|_{C^l(K)} \leq M_{l,K} \left( \| \varphi_i' - \varphi' \|_{C^l(K)} + \| g_i^{-1} - g^{-1} \|_{C^l(K)} \right). \]
As the right-hand side converges to zero, we see that $\varphi_i$ converges to $\varphi$ in $C^l$ for any $l$ on $K$. A similar argument shows the convergence of connections.
\end{step}

\begin{step}[The limiting configuration] 
Passing to the limit in equation \eqref{eqn:vortex1}, we see that $f \colon \Sigma \setminus D \to \R$ is given by
\[  f = \log | \varphi' |^2, \]
which, by $\varphi = e^{-f/2} \varphi'$, is equivalent to $| \varphi | = 1$.  Since we also have $\del_A \varphi = 0$, Lemma \ref{lem:flatness} implies that $\nabla_A \varphi = 0$ and $F_A = 0$ on $\Sigma \setminus D$.
\end{step} 

\begin{step}[Convergence of measures] 
It remains to show that 
\[  \frac{i}{2\pi} \Lambda F_{A_i} \to  \sum_{j = 1}^d \delta (x_j) \]
in the sense of measures, or, equivalently, that for any small disc $B$ around $x_j$, 
\[ \lim_{i \to \infty} \int_{B} \frac{i}{2\pi} F_{A_i} = k, \]
where $k$ is the  multiplicity of the section $\varphi'$ at $x_j$. Choose local coordinates on $B$ together with a unitary trivialisation of $L$. Then $A_i$ is of the form $A_i = d + a_i$ for $a_i \in \Omega^1(B, i\R)$ and the curvature is $F_{A_i} = da_i$. By Stokes' theorem,
\begin{equation} \label{eqn:vortex3} \lim_{i \to \infty} \int_{B} \frac{i}{2\pi} F_{A_i} = \lim_{i \to \infty} \int_{B} \frac{i}{2\pi} da_i = \lim_{i \to \infty} \int_{\partial B} \frac{i}{2\pi} a_i = \int_{\partial B} \frac{i}{2\pi} a, 
\end{equation}
where $a \in \Omega^1(B \setminus \{ x_j \}, i\R)$ is the one-form corresponding to the singular connection $A = d + a$. By Lemma \ref{lem:flatness}, the integral on the right-hand side is the degree of the limiting section $\varphi$ around $x_j$. Since $\varphi'$ and $\varphi$ differ by a non-zero real function on $B \setminus \{ x_j \}$, their degrees around $x_j$ are the same and equal to the multiplicity of $\varphi'$ at $x_j$. 
\end{step}
\end{proof}

We now turn to Theorem \ref{thm:multi}. 
%Although the convergence statement fits into the more general setting of Theorem \ref{thm:generalised}, we present the proof in this special case first in order to keep the notation simple.
 For simplicity, we write $\del_{A}$ instead of $\del_{BA}$.

\begin{proof}[Proof of Theorem \ref{thm:multi}]
%The main idea is the same as before. First we find a subsequence which converges after rescaling and applying complex gauge transformations $g_i$. Then, using Proposition \ref{prop:bounds} we extract a subsequence of $g_i$ converging with all derivatives outside a singular set.
To avoid double upper indices, denote $\alpha_i = \varphi_i^1$ and $\beta_i = \varphi_i^2$. Integrating the last equation of \eqref{eqn:multi}, we obtain
\begin{equation} \label{eqn:compactness0} \| \alpha_i \|_{L^2}^2 - \| \beta_i \|_{L^2}^2 = - 2 \pi \epsilon_i  \deg L, 
\end{equation}
which together with
\[  \| \alpha_i \|_{L^2}^2 + \| \beta_i \|_{L^2}^2 =  \| \varphi_i \|_{L^2}^2 = 1 \]
implies that for sufficiently large $i$ neither $\alpha_i$ nor $\beta_i$ is identically zero. 

\setcounter{step}{0}

\begin{step}[Convergence modulo $\mathcal{G}^c$] We claim that there are sequences $g_i \in \mathcal{G}^c$ and $\lambda_i > 0$ such that after passing to a subsequence $(g_i A_i, \lambda_i g_i \alpha_i, \lambda_i g_i^{-1} \beta_i)$ converges in $C^{\infty}$ on $\Sigma$ to a triple $(A', \alpha', \beta')$. The limiting sections $\alpha'$ and $\beta'$ are not identically zero and satisfy
\[ \del_{A'} \alpha' = 0 \qquad \textnormal{and} \qquad \del_{A'} \beta' = 0. \] 

By the same argument as in the proof of Theorem \ref{thm:vortex}, we can find $g_i \in \mathcal{G}^c$ such that, after passing to a subsequence, $A_i' = g_i A_i$ converges in $C^{\infty}$ to a connection $A'$. We want to rescale the corresponding sequences of sections $g_i \alpha_i$ and $g_i^{-1} \beta_i$ so that they also converge, and that the limiting sections are non-zero. First, by replacing the sequence $g_i$ by $\mu_i g_i$, where $\mu_i$ are constant complex gauge transformations given by
\[ \mu_i = \sqrt{ \frac{ \| g_i^{-1} \beta_i \|_{L^2}   }{ \| g_i \alpha_i \|_{L^2}  } }, \] 
we may assume that $\| g_i \alpha_i \|_{L^2} = \| g_i^{-1} \beta_i \|_{L^2}$. After changing the original sequence by real gauge transformations we can assume that $g_i = e^{f_i/2}$ for smooth functions $f_i \colon \Sigma \to \R$.  

Set  $\lambda_i = \| g_i \alpha_i \|_{L^2}^{-1} $ and consider the rescaled sequences
\[ \alpha_i' = \lambda_i g_i \alpha_i \qquad \textnormal{and} \qquad \beta_i' = \lambda_i g_i^{-1} \beta_i. \]
We have $\| \alpha_i' \|_{L^2} = \| \beta_i' \|_{L^2} = 1$ so as in the proof of Theorem \ref{thm:vortex} we can choose subsequences (denoted for simplicity by the same symbols) that converge in $C^{\infty}$ to sections $\alpha'$ and $\beta'$ say. They are holomorphic with respect to $A'$ and satisfy
\[ \| \alpha' \|_{L^2} = \lim_{i \to \infty} \| \alpha_i' \|_{L^2} = 1, \qquad  \| \beta' \|_{L^2} = \lim_{i \to \infty} \| \beta_i' \|_{L^2} = 1. \]
\end{step}

\begin{step}[An upper bound for $\lambda_i$] Next, we show that the sequence $\lambda_i$ is bounded above. Assume by contradiction that after passing to a subsequence $\lambda_i \to \infty$. Then
\begin{equation}  | \alpha_i | | \beta_i | = \lambda_i^{-2} | \alpha_i' | | \beta_i' | \leq \lambda_i^{-2} \sup_{i} \| \alpha_i ' \|_{L^{\infty}} \| \beta_i' \|_{L^{\infty}} .
\end{equation}
Since $\alpha_i' \to \alpha'$ and $\beta_i' \to \beta'$ uniformly, we conclude that
\begin{equation}\label{eqn:compactness1}
\lim_{i \to \infty} \| | \alpha_i | | \beta_i | \|_{L^{\infty}} = 0. 
\end{equation}

We will argue that this cannot happen. Set $\psi_i = | \beta_i |^2$ and $Q_i = | \beta_i' |^2$. Differentiating  $ \lambda_i^{2} \psi_i = e^{f_i} Q_i $ twice, we arrive at
\[ \lambda_i^2 \nabla \psi_i = e^{f_i} Q_i \nabla f_i + e^{f_i} \nabla Q_i, \]
and
\[
 \begin{split}  \lambda_i^2 \Delta \psi_i &= \Delta( e^{f_i} ) Q_i - \left\langle \nabla ( e^{f_i} ), \nabla Q_i \right\rangle + e^{f_i} \Delta Q_i, \\
&= e^{f_i} Q_i \Delta f_i - e^{f_i} Q_i | \nabla f_i |^2 - e^{f_i}  \left\langle \nabla f_i , \nabla Q_i \right\rangle+ e^{f_i} \Delta Q_i \\
&= e^{f_i} Q_i \Delta f_i - \left\langle \nabla f_i , e^{f_i} Q_i \nabla f_i + e^{f_i} \nabla Q_i \right\rangle + e^{f_i} \Delta Q_i  \\
&= e^{f_i} Q_i \Delta f_i - \left \langle \nabla f_i, \lambda_i^2 \nabla \psi_i \right\rangle + e^{f_i} \Delta Q_i.
\end{split}
\]
Let $x_i \in \Sigma$ be a global maximum of $\psi_i$, so that $\nabla \psi_i(x_i) = 0$ and $\Delta \psi_i (x_i) \geq 0$. The calculation above implies that
\[ Q_i (x_i) \Delta f_i(x_i) \geq - \Delta Q_i(x_i). \]
We would like to conclude that $\Delta f_i$ is bounded below for all sufficiently large $i$. Consider $i$ large enough so that $Q_i$ is sufficiently close to $Q = | \beta' |^2$ in the $C^2$ norm. Since $\beta'$ is holomorphic and not identically zero, it vanishes at finitely many points. Therefore, for any $\delta > 0$ there is a small open neighbourhood, $V_{\delta}$ say, of the zero set of $\beta'$ such that $Q_i(x) \geq \delta$ whenever $x \in \Sigma \setminus V_{\delta}$ and $i$ is large. It follows that if the sequence $x_i$ is isolated from the zero set of $\beta'$, then there exists $\delta > 0  $ such that 
\[ \Delta f_i(x_i) \geq - \frac{1}{\delta} \sup_i \| \Delta Q_i \|_{L^{\infty}}.  \]
On the other hand, assume that after passing to a subsequence, $x_i \to x$, where $\beta'(x) = 0$. If $x$ is a simple root of $\beta'$, then in local holomorphic coordinates centred around $x$ we have $Q(z) = c |z|^2 + O(|z|^3)$ for some $c > 0$ and as a result $\Delta Q (x) < 0$. It follows that $\Delta Q_i (x_i) < 0$ for $i$ sufficiently large and by \eqref{eqn:compactness2}, we must have that $Q_i(x_i) \neq 0$ since
\[ 0 \leq \lambda_i^2 \Delta \psi_i (x_i)  = e^{f_i(x_i)} Q_i(x_i) \Delta f_i(x_i) + e^{f_i(x_i)} \Delta Q_i(x_i). \]
Thus, dividing both sides by $e^{f_i(x_i)} Q_i(x_i)$ results in the inequality $\Delta f_i(x_i) \geq 0$. In the general situation, the section $\beta'$ vanishes to the order $k$, say. In this case, locally $Q^{1/k}(z) = c |z|^2 + O(|z|^3)$ for some $c > 0$. We can therefore, at least locally, replace $\psi_i$ by $\psi_i^{1/k}$ (note that $\psi_i(x_i) \to \infty$, so $\psi_i^{1/k}$ is a smooth function around $x_i$), $f_i$ by $f_i / k$, and $Q_i$ by $Q_i^{1/k}$ and repeat the previous argument to obtain $\Delta f_i(x_i) \geq 0$. 

From the lower bound for $\Delta f_i(x_i)$ we easily arrive at a contradiction. Indeed, since $A_i' = g_i( A_i)$, the third equation of \eqref{eqn:multi} can be written in the form
\[ \epsilon_i^2  \Lambda i F_{A_i'} = \epsilon_i^2 \Lambda i F_{A_i} + \epsilon_i^2 \Delta f_i = - |\alpha_i|^2 + |\beta_i|^2 + \epsilon_i^2 \Delta f_i, \]
or after rearranging,
\begin{equation} 
 \label{eqn:compactness1.2}
 | \beta_i |^2 =  \epsilon_i^2 \Lambda i F_{A_i'} - \epsilon_i^2 \Delta f_i + |\alpha_i |^2. 
\end{equation}
Since $\Delta f_i(x_i)$ is bounded below and $A_i'$ converges in $C^{\infty}$, at $x_i$ we have
\begin{equation}
  \label{eqn:compactness1.5}
| \beta_i(x_i) |^2 \leq C \epsilon_i^2 + | \alpha_i(x_i)|^2.
\end{equation}
Equation \eqref{eqn:compactness0} implies that for large $i$ 
\[
  | \beta_i(x_i) | = \| \beta_i (x_i) \|_{L^{\infty}} \geq C \| \beta_i \|_{L^2} = C\left(\frac{1}{2}+ \pi \epsilon_i \deg L\right) \geq \frac{C}{4} > 0,
\]
so $| \beta_i(x_i) |$ is separated from zero and \eqref{eqn:compactness1} forces $| \alpha_i(x_i) |$ to converge to zero. Then \eqref{eqn:compactness1.5} yields $| \beta_i(x_i) | \to 0$, which is a contradiction. This finishes the proof that the sequence $\lambda_i$ is bounded above.
\end{step}

\begin{step}[Convergence of $\lambda_i$] The next step is to show that the sequence $\lambda_i$ is separated from zero. This follows from the inequality
\[ \lambda_i^2 ( | \alpha_i|^2 + |\beta_i|^2 ) \geq 2 \lambda_i^2 | \alpha_i | |\beta_i| = 2  | \lambda_i e^{f_i/2} \alpha_i | | \lambda_i e^{-f_i/2} \beta_i | = 2 | \alpha_i' | | \beta_i' |. \]
Integrating over $\Sigma$ and taking the square root, we obtain
\[ \lambda_i \geq \sqrt{ 2 \int_\Sigma |\alpha_i' | | \beta_i' | } \]
Since $\alpha_i' \to \alpha'$ and $\beta_i' \to \beta'$ uniformly on $\Sigma$ and the limiting sections are holomorphic and non-zero, the integral on the right-hand side is bounded below by a positive number. Since $\lambda_i$ is bounded above and separated from zero, after passing to a subsequence we can assume that $\lambda_i$ converges to a positive number, $\lambda$ say. Therefore, 
\[ e^{f_i/2} \alpha_i = \lambda_i^{-1} \alpha_i' \to \lambda^{-1} \alpha', \]
\[ e^{-f_i/2} \beta_i = \lambda_i^{-1} \beta_i' \to \lambda^{-1} \beta' \]
 and after  rescaling $\alpha'$, $\beta'$ we may assume that $\lambda_i = \lambda = 1$.
\end{step}

\begin{step}[Convergence outside the singular set]
 Let $D$ be the union of the zero sets of $\alpha'$ and $\beta'$. Since the sections are holomorphic and not identically zero, $D$ is empty or consists of finitely many points.  We claim that after passing to a subsequence, $f_i$ converges in $C^{\infty}_{\loc}(\Sigma \setminus D)$. To see this, we translate the third equation of \eqref{eqn:multi} for the triple $(A_i, \alpha_i, \beta_i)$ into a partial differential equation for $f_i$:
\[ \begin{split} 
\epsilon_i^2 \Lambda i F_{A_i'} &= \epsilon_i^2 \Lambda i F_{A_i} + \epsilon_i^2 \Delta f_i = - |\alpha_i|^2 + |\beta_i|^2 + \epsilon_i^2 \Delta f_i \\
&= -  e^{-f_i} | \alpha_i' |^2 +  e^{f_i} | \beta_i' |^2 + \epsilon_i^2 \Delta f_i, 
\end{split} \]
or after rearranging
\begin{equation} \label{eqn:compactness2}
 \epsilon_i^2 \Delta f_i - P_i e^{-f_i} + Q_i e ^{f_i} - w_i = 0, 
 \end{equation}
where $P_i = | \alpha_i' |^2$, $Q_i = | \beta_i' |^2$, and $w_i = \epsilon_i^2 \Lambda i F_{A_i'}$. Note that $\epsilon_i \to 0$ and the functions $P_i$, $P_i$, and $w_i$ converge in $C^{\infty}$ to, respectively, $P = |\alpha'|^2$, $Q = | \beta' |^2$, and $w = 0$. Furthermore, $P > 0$ and $Q > 0$ on $\Sigma \setminus D$.

It follows from Proposition \ref{prop:bounds} and Remark \ref{rem:bounds} that the sequence $f_i$ and its derivatives are uniformly bounded on every compact subset of $\Sigma \setminus D$. Therefore, we can choose a subsequence of $f_i$ which converges in $C^{\infty}_{\loc}(\Sigma \setminus D)$. The limit is a smooth function $f \colon \Sigma \setminus D \to \R$. Let $g = e^{f/2}$ be the corresponding complex gauge transformation. Now set $(A, \alpha, \beta) = ( g^{-1}( A' ), g^{-1} \alpha', g\beta')$. The triple is well-defined on $\Sigma \setminus D$ and $(A_i,  \alpha_i, \beta_i) \to (A, \alpha, \beta)$ in $C^{\infty}_{\loc}$ on $\Sigma \setminus D$.
\end{step}

\begin{step}[The limiting configuration]
After passing to the limit $i \to \infty$ in \eqref{eqn:compactness2}, we get
\[ - e^{-f} |\alpha'|^2 + e^f | \beta'|^2 = 0, \]
or equivalently $e^f = | \alpha'| / |\beta'|$. Therefore, the limiting configuration $\varphi = (\alpha, \beta)$ satisfies
\[- |\alpha|^2 + |\beta|^2 = - e^{-f} | \alpha'|^2 + e^f | \beta'|^2 = 0. \]
The remaining equations $\del_A \alpha = 0$, $\del_A \beta = 0$, and $\alpha \beta = 0$ are obtained from passing to the limit in the corresponding equations for $A_i, \alpha_i$, and $\beta_i$.  Moreover, 
\[ | \varphi |^4 = \left( |\alpha|^2 + |\beta|^2 \right)^2 = \left( e^{-f} | \alpha' |^2 + e^f | \beta'|^2 \right)^2 = 4 |\alpha'|^2 | \beta'|^2 \]
extends to a smooth function on $\Sigma$ whose zeroes are the points in $D$. At a point $x \in D$, the norm $| \varphi | = 2 \sqrt{ | \alpha' | | \beta' |}$ vanishes to the order 
\[
\frac{1}{2} \left( \mathrm{ord}_x (\alpha') + \mathrm{ord}_x (\beta') \right)
\]
where $\mathrm{ord}_x (\alpha')$ and $\mathrm{ord}_x (\beta')$ are the orders of vanishing of $\alpha'$ and $\beta'$ at $x$.

Assume now that $\rank E = 2$. Let $A^2 = A \otimes A$ denote the tensor product connection on $L^2 = L \otimes L$. We will show that the pair $(L^2, A^2)$ is trivial as a bundle with a connection. It will follow then that $A$ is flat and has holonomy contained in $\Z_2$. Since $\rank E = 2$ and $\alpha \beta = 0$,  we have the following exact sequence over $\Sigma \setminus D$:
\begin{equation*}
\label{eq:shortexact}
\begin{tikzcd}
0 \arrow{r} & L^* \otimes K^{-1/2} \arrow{r}{\alpha} & E \arrow{r}{\beta} & L^* \otimes K^{1/2} \arrow{r} & 0,
\end{tikzcd}
\end{equation*}
which yields a bundle isomorphism over $\Sigma \setminus D$
\[ \psi_{\alpha\beta} \colon  L^{-2} \otimes K^{-1/2} \otimes K^{1/2} \longrightarrow \det E. \]
Both line bundles $\det E$ and $K^{-1/2} \otimes K^{1/2}$ are trivial as bundles with connections. Thus, $\psi_{\alpha\beta}$ is a section of $L^2$ over $\Sigma \setminus D$ which satisfies
\[ \del_{A^2} \psi_{\alpha\beta} = 0, \qquad | \psi_{\alpha\beta} | = \frac{ | \alpha | }{|\beta|} = 1 \]
By Lemma \ref{lem:flatness}, $\nabla_{A^2} \psi_{\alpha\beta} = 0$ and $(L^2, A^2)$ is trivial as a bundle with a connection.

If $\rank E > 2$, let $F$ be the complex subbundle of $\restr{E}{\Sigma \setminus D}$ spanned by the subspaces $\im \alpha$ and $(\ker \beta)^{\perp}$, interpreting $\alpha$ and $\beta$ as in the short exact sequence \eqref{eq:shortexact}. 
Equip $F$ with the unitary connection $\tilde{B}$ obtained from $B$ using the orthogonal projection $\pi_F \colon E \to F$, and let $\del_{\tilde{B}} = (\nabla_{\tilde{B}})^{0,1} = \pi_F \del_B$ be the induced holomorphic structure on $F$.
Note that $F \subset E$ is not, in general, a holomorphic subbundle. 
Nevertheless, $\alpha$ and $\beta$ descend to smooth sections
\begin{equation*}
\label{eq:reducedsections}
 \alpha \in \Gamma(\Sigma \setminus D, F  \otimes  L  \otimes K^{1/2}), \qquad \beta \in \Gamma(\Sigma \setminus D, F^* \otimes L^* \otimes K^{1/2}) \]
which are nowhere zero and satisfy $\alpha \beta = 0$; we claim that these sections are in fact holomorphic with respect to $\del_{\tilde{B}}$. This is clear for $\alpha$ since $\del_{\tilde{BA}} \alpha = \pi_F \del_{BA} \alpha = 0$. To prove that $\beta$ is also holomorphic, it is enough to show that for every locally defined section $s$ of $F$ satisfying $\del_{\tilde{B}} s = 0$ we have $\del_A ( \beta(s)) = 0$ (thinking of $\beta$ as a homomorphism from $F$ to $L^* \otimes K^{1/2}$).
The condition $0 = \del_{\tilde{B}} s = \pi_F \del_B s$ implies that $\del_B s$ takes values in $\ker \beta$ and since $\beta \colon E \to L^* \otimes K^{1/2}$ is holomorphic, we have $\del_A(\beta(s)) = \beta(\del_{B} s) = 0$. 

Thus, $\alpha$ and $\beta$ in \eqref{eq:reducedsections} are holomorphich with respect to $\del_{\tilde{B}}$ and so they fit into the short exact sequence \eqref{eq:shortexact} with $F$ in place of $E$.
The previous argument shows that there is a nowhere vanishing section $\psi_{\alpha \beta} \in \Gamma(\Sigma, \setminus D, \det F \otimes L^2)$ satisfying $\del_{BA} \psi_{\alpha\beta} = 0$ and $| \psi_{\alpha\beta}| =1$. Here $F$ is equipped with a connection obtained from $B$ by the orthogonal projection $E \to F$, and $\det F$ has the induced connection. Lemma \ref{lem:flatness} implies that $\nabla_{BA} \psi_{\alpha\beta} = 0$ and $\det F \otimes L^2$ is trivial over $\Sigma \setminus D$. In particular, there is a well-defined square root $(\det F)^{1/2}$ and the connection on $(\det F)^{1/2} \otimes L$ induced from $B$ and $A$ is flat with holonomy contained in $\Z_2$. 
\end{step}

\begin{step}[Convergence of measures]
We claim that if $x$ is a point in $D$ at which the sections $\alpha'$ and $\beta'$ vanish to the order $k = \mathrm{ord}_x (\alpha')$ and $l = \mathrm{ord}_x (\beta')$ respectively, and $B$ is a small disc around $x$, then
\[ \lim_{i \to \infty} \int_B   \frac{i}{2\pi} \Lambda F_{A_i}  = \frac{k-l}{2}. \]
This is proved in the same way as the corresponding statement in Theorem \ref{thm:vortex}, except that now Lemma \ref{lem:flatness} is applied to the connection $A^2$ on $L^2$ and section $\psi_{\alpha\beta}$ (under the assumption that $\rank E = 2$; in the higher rank case we need to twist $L^2$ by $\det F$). The degree of the restriction of $\psi_{\alpha\beta}$ to $\partial B$ is
\[ \mathrm{deg}\left( \psi_{\alpha\beta} \right) = \mathrm{deg} \left( \frac{ \alpha }{ |\alpha|}\right) - \mathrm{deg}\left( \frac{\beta}{ |\beta|}\right) = k - l, \]
and the factor $1/2$ enters because we consider $A$ instead of $A^2$. This shows that $D$ is the set underlying the divisor
\[
   D = \sum_x \left( \mathrm{ord}_x (\alpha') - \mathrm{ord}_x (\beta') \right)x
\]
and $i \Lambda F_A = \frac{1}{2} \delta_D$ as measures. In particular, $D$ has degree $2d$.
\end{step}
\end{proof}

\begin{proof}[Proof of Theorem \ref{thm:nsw}]
Our assumptions guarantee that all irreducible solutions to the Seiberg--Witten equations with multiple spinors are gauge-equivalent to configurations pulled back from $\Sigma$ satisfying \eqref{eqn:multi} with $\epsilon = 1$ \cite[Proposition 3.10]{doan}. The bundle $L$ is necessarily pulled back from a bundle over $\Sigma$. Let $\epsilon_i = \| \Psi_i \|_{L^2}^{-1}$ and $\Psi_i' = \epsilon_i \Psi_i$. Then the rescaled sequence $(A_i, \Psi_i')$ satisfies $\| \Psi_i' \|_{L^2} = 1$ and equations \eqref{eqn:multi} with $\epsilon = \epsilon_i$. The statements about the convergence of $(A_i, \Psi_i')$ and the structure of $Z$ follow now from Theorem \ref{thm:multi}. 

It remains to show the convergence of currents. Every $\eta \in \Omega^1(Y)$ is of the form
\[ \eta = f (t,x) dt + \xi (t), \]
where $f \in C^{\infty}(Y)$, $(t,x)$ denote the product co-ordinates on $S^1 \times \Sigma$, and $\xi$ is an $S^1$-family of one-forms $\xi(t) \in \Omega^1(\Sigma)$. The forms $F_{A_i}$ are pulled back from $\Sigma$, so
\[ \int_Y \frac{i}{2\pi} F_{A_i} \wedge \eta = \int_{S^1} \int_{\Sigma} \frac{i}{2\pi} F_{A_i} f(t,x) dt = \int_{\Sigma} \frac{i}{2\pi} g(x) F_{A_i}, \]
where $g(x) = \int_{S^1} f(t,x) dt$. Passing to the limit $i \to \infty$ and using the convergence of measures $(i/2\pi) \Lambda F_{A_i} \to \delta_D / 2$ proved in Theorem \ref{thm:multi}, we arrive at
\[ \lim_{i \to \infty} \int_Y \frac{i}{2\pi} F_{A_i} \wedge \eta = \frac{1}{2} \sum_k m_k g(x_k) = \frac{1}{2} \sum_k m_k \int_{S^1} f(t,x) dt = \frac{1}{2} \sum_k m_k \int_{S^1 \times \{ x_k \}} \eta, \]
which is the equality that we wanted to prove.

It remains to address the statement made in point $(3)$ of the theorem about the weights associated with the connected components of $Z$.
If the singular set $Z$ is a disjoint union $\bigcup_k Z_k$ of connected one-dimensional submanifolds, the weight prescribed by Taubes to $Z_k$ is the highest number $N_k$ such that $|\Psi|$ vanishes to the order $N_k / 2$ along $Z_k$ \cite{taubes}. In our case, $Z_k = S^1 \times \{ x_k \}$, $N_k$ is an integer, and by point $(3)$ of Theorem \ref{thm:multi}, we have $|m_k| \leq N_k$. On the other hand, $m_k$ is the coefficient of the current $\frac{1}{2} Z$ and equals the weight prescribed to $Z_k$ by Haydys: the degree of the section of $L^2$ induced from $\Psi$ restricted to a small loop linking $Z_k$  \cite{haydys}; see step $5$ in the proof of Theorem \ref{thm:multi}.

\end{proof}

Finally, we prove Theorem \ref{thm:generalised}. The following inequality will be useful.

\begin{lem} \label{lem:inequality}
Let $a$, $b$, $x$, and $y$ be positive numbers. Then for all $\xi \in (0, \infty)$
\[ \xi^{-a} x + \xi^{b} y \geq K x^{\frac{b}{a+b}} y^{\frac{a}{a+b}} \]
where $K$ is positive and depends only on $a$ and $b$. 
\end{lem}
\begin{proof}
Consider the function $f \colon (0, \infty) \to (0, \infty)$ given by 
\[ f(\xi) = \xi^{-a} x + \xi^{b} y. \]
It attains a global minimum at
\[ \xi_0 = \left( \frac{ax}{by} \right)^{\frac{1}{a+b}}. \]
and the value of $f$ at $\xi_0$ is
\[ f(\xi_0) = K x^{\frac{b}{a+b}} y^{\frac{a}{a+b}}  \]
with  
\[ K =  (a/b)^{\frac{b}{a+b}} + (b/a)^{\frac{a}{a+b}} \]
\end{proof}

\begin{proof}[Proof of Theorem \ref{thm:generalised}]  First, order the weights according to their signs:
\[ (k_1, \ldots, k_N) = (a_1, \ldots, a_n, -b_1, \ldots, -b_m), \]
where $n+m = N$ and $a_j  > 0$ and $b_j > 0$. Denote 
\[ \alpha_i = ( \alpha^1_i, \ldots, \alpha^n_i) = (\varphi^1_i, \ldots, \varphi^n_i), \]
\[ \beta_i = ( \beta^1_i, \ldots, \beta^m_i) = (\varphi^{n+1}_i, \ldots, \varphi^N_i). \]
When all the weights are positive and $\tau < 0$, or when one of the sequences $\| \alpha_i \|_{L^{\infty}}$ or $\| \beta_i \|_{L^{\infty}}$ converges to zero, the proof is the same as for Theorem \ref{thm:vortex}. Let us focus on the more difficult case when the weights are of mixed signs and $\| \alpha_i \|_{L^{\infty}}$, $\| \beta_i \|_{L^{\infty}}$ are bounded below by a positive number. 

As in the proof of Theorem \ref{thm:multi}, we find sequences of $\lambda_i > 0$ and complex gauge transformations $g_i \in \mathcal{G}^c$ such that the sequence 
\[ (A_i', \alpha_i', \beta_i') = g_i(A_i, \lambda_i \alpha_i, \lambda_i \beta_i)\]
 converges uniformly with all derivatives on $\Sigma$ to a limit $(A', \alpha', \beta')$. In order to do this, first choose $g_i$ so that $g_i (A_i)$ converges. After changing $g_i$ by constant complex gauge transformations we may assume that $\| g_i (\alpha_i) \|_{L^2} = \| g_i ( \beta_i) \|_{L^2}$. Indeed, for each $i$ such a constant gauge transformation is given by a number $\mu > 0$ satisfying
\[ \sum_{j=1}^n \mu^{a_j} \| \alpha_i^j \|_{L^2}^2 - \sum_{j=1}^{m} \mu^{-b_j} \| \beta_i^j \|_{L^2}^2 = 0. \]
Since the left-hand side diverges to $\infty$ when $\mu \to \infty$, and to $- \infty$ when $\mu \to 0$, such $\mu$ exists. Then we choose the scaling constants $\lambda_i$ so that
\[ \| \lambda_i g_i ( \alpha_i ) \|_{L^2} = \| \lambda_i g_i ( \beta_i) \|_{L^2} = 1, \]
which, after passing to a subsequence, guarantees the existence of non-zero limiting sections $\alpha' = ((\alpha')^1, \ldots, (\alpha')^n)$ and $\beta' = ((\beta')^1, \ldots, (\beta')^m)$. After changing the original sequence by a sequence of unitary gauge transformations, we may assume that $g_i = e^{f_i/2}$ for smooth functions $f_i \colon \Sigma \to \R$. 

The next step is to show that $\lambda_i$ is bounded above and separated from zero. To prove the latter observe that at least one of the sections $(\alpha')^i$ is non-zero. Without loss of generality assume that $(\alpha')^1\neq 0$. Likewise, we may assume that $(\beta')^1 \neq 0$. Then 
\[ \lambda_i^2 \left( | \alpha_i |^2 + | \beta_i |^2 \right) \geq | \lambda_i \alpha_i^1 |^2 + | \lambda_i \beta_i^1 |^2 =  e^{-a_1 f_i } | (\alpha_i')^1 |^2 + e^{ b_1 f_i } | (\beta_i')^1 |^2. \]
Applying Lemma \ref{lem:inequality} to the right-hand side, integrating over $\Sigma$, and using the $L^2$-bound for $\varphi_i = ( \alpha_i, \beta_i)$, we obtain  
\[ C \lambda_i^2 \geq K \int_{\Sigma} | (\alpha_i')^1 |^{\kappa_1} | (\beta_i')^1 |^{\kappa_2} \]
for some positive exponents $\kappa_1$ and $\kappa_2$ depending on $a_1$ and $b_1$. Now the right-hand side of the inequality is bounded below by a positive constant, because of the convergence $(\alpha_i')^1 \to (\alpha')^1$ and $(\beta_i')^1 \to (\beta')^1$ and the assumption that the limiting sections are non-zero. This shows that the sequence $\lambda_i$ is separated from zero. 

An upper bound is found in the same way as in the proof of Theorem \ref{thm:multi}. Let
\[ u_i = \sum_{j=1}^n a_j | \alpha_i^j |^2 \qquad \textnormal{and} \qquad v_i = \sum_{j=1}^m b_j | \beta_i^j|^2. \]
Suppose without loss of generality that $\tau \leq 0$. We have, as in \eqref{eqn:compactness1.2},
\[ v_i = \epsilon_i^2 \Lambda i F_{A_i'} - \epsilon_i^2 \Delta f_i + u_i + \tau. \]
Using the maximum principle, one shows that at the point $x_i$ where $v_i$ attains a global maximum, there is a lower bound for $\Delta f_i(x_i)$. As in the situation of Theorem \ref{thm:multi}, if $\lambda_i$ is unbounded, then  $ u_i v_i \to 0$ uniformly. It follows that $u_i(x_i) \to 0$ and passing to the limit $i \to \infty$ we arrive at a contradiction $v_i(x_i) \to 0$. This shows that $\lambda_i$ must be bounded above. Since it is also separated from zero, after passing to a subsequence it may be assumed to be convergent and we may as well assume that $\lambda_i = 1$. 

The function $f_i$ satisfy the partial differential equation
\[ \epsilon_i ^2 \Delta f_i - \sum_{j=1}^n P_i^j e^{- a_j f_i} + \sum_{j=1}^m Q_i^j e^{b_j f_i} - w_i = 0, \]
where $P_i^j = a_j | g_i (\alpha_i^j) |^2$, $Q_i^j = b_j | g_i ( \beta_i^j) |^2$, and $w_i = \epsilon_i^2 i \Lambda F_{g_i(A_i)}$. We have
\[ P_i^j \longrightarrow P^j, \quad Q_i^j \longrightarrow Q^j  , \quad w_i \longrightarrow 0 \quad \textnormal{in } C^{\infty},\]
where $P^j = a_j | (\alpha')^j|^2$ and $Q^j = b_j | (\beta')^j |^2$. Let $D $ be the union of the zero sets of the limiting sections $\alpha'$ and $\beta'$. Since they are not identically zero, $D$ is a (possibly empty) finite subset of $\Sigma$.  Over $\Sigma \setminus D$ we have 
\[ P^1 + \ldots + P^n > 0, \qquad \textnormal{and} \qquad Q^1 + \ldots + Q^m > 0. \]
Therefore, by Proposition \ref{prop:bounds} and Remark \ref{rem:bounds} we can bound  the the functions $f_j$ and their derivatives on any compact subset of $\Sigma \setminus D$, and consequently, extract a subsequence converging smoothly to a function $f$ on $\Sigma \setminus D$. As before, this leads to the smooth convergence of $(A_i, \alpha_i, \beta_i)$ to $e^{-f/2}(A', \alpha', \beta')$ on $\Sigma \setminus D$.
 \end{proof}

\newpage
\bibliography{adiabatic.bib} 

\begin{thebibliography}{CGMiRS}

\bibitem[BC]{bismut}
J.-M. Bismut and J.~Cheeger.
\newblock {$\eta$}-invariants and their adiabatic limits.
\newblock {\em J. Amer. Math. Soc.}, 2(1):33--70, 1989.

\bibitem[BGP]{bradlow2}
S.~B. Bradlow and O.~Garc{\'{\i}}a-Prada.
\newblock Non-abelian monopoles and vortices.
\newblock In {\em Geometry and physics ({A}arhus, 1995)}, volume 184 of {\em
  Lecture Notes in Pure and Appl. Math.}, pages 567--589. Dekker, New York,
  1997.

\bibitem[Bra]{bradlow}
S.~B. Bradlow.
\newblock Vortices in holomorphic line bundles over closed {K}\"ahler
  manifolds.
\newblock {\em Comm. Math. Phys.}, 135(1):1--17, 1990.

\bibitem[BW]{bw}
J.~A. Bryan and R.~Wentworth.
\newblock The multi-monopole equations for {K}\"ahler surfaces.
\newblock {\em Turkish J. Math.}, 20(1):119--128, 1996.

\bibitem[CGMiRS]{cgms}
K.~Cieliebak, A.~R. Gaio, I.~Mundet~i Riera, and D.~A. Salamon.
\newblock The symplectic vortex equations and invariants of {H}amiltonian group
  actions.
\newblock {\em J. Symplectic Geom.}, 1(3):543--645, 2002.

\bibitem[Doa]{doan}
A.~Doan.
\newblock {Seiberg-Witten monopoles with multiple spinors on a surface times a
  circle}.
\newblock {\em eprint arXiv:1701.07942}, 2017.

\bibitem[DS1]{donaldson-segal}
S.~Donaldson and E.~Segal.
\newblock Gauge theory in higher dimensions, {II}.
\newblock In {\em Surveys in differential geometry. {V}olume {XVI}. {G}eometry
  of special holonomy and related topics}, volume~16 of {\em Surv. Differ.
  Geom.}, pages 1--41. Int. Press, Somerville, MA, 2011.

\bibitem[DS2]{ds}
S.~Dostoglou and D.~A. Salamon.
\newblock Self-dual instantons and holomorphic curves.
\newblock {\em Ann. of Math. (2)}, 139(3):581--640, 1994.

\bibitem[DW]{doan-walpuski}
A.~Doan and T.~Walpuski.
\newblock {On the existence of harmonic $\mathbb{Z}_2$ spinors}.
\newblock {\em eprint arXiv:1710.06781}, 2017.

\bibitem[Fin]{fine}
J.~Fine.
\newblock Constant scalar curvature {K}\"ahler metrics on fibred complex
  surfaces.
\newblock {\em J. Differential Geom.}, 68(3):397--432, 2004.

\bibitem[GaP1]{garcia-prada1}
O.~Garc\'\i~a Prada.
\newblock Invariant connections and vortices.
\newblock {\em Comm. Math. Phys.}, 156(3):527--546, 1993.

\bibitem[GaP2]{garcia-prada2}
O.~Garc\'\i~a Prada.
\newblock A direct existence proof for the vortex equations over a compact
  {R}iemann surface.
\newblock {\em Bull. London Math. Soc.}, 26(1):88--96, 1994.

\bibitem[GS]{gs}
A.~R.~P. Gaio and D.~A. Salamon.
\newblock Gromov-{W}itten invariants of symplectic quotients and adiabatic
  limits.
\newblock {\em J. Symplectic Geom.}, 3(1):55--159, 2005.

\bibitem[{Hay}1]{haydys}
A.~{Haydys}.
\newblock {The infinitesimal structure of the blow-up set for the
  Seiberg-Witten equation with multiple spinors}.
\newblock {\em eprint arXiv:1607.01763}, 2016.

\bibitem[Hay2]{haydys4}
A.~Haydys.
\newblock {$G_2$-instantons and Seiberg-Witten monopoles}.
\newblock {\em eprint arXiv:1703.06329}, 2017.

\bibitem[HJS]{hjs}
M.-C. Hong, J.~Jost, and M.~Struwe.
\newblock Asymptotic limits of a {G}inzburg-{L}andau type functional.
\newblock In {\em Geometric analysis and the calculus of variations}, pages
  99--123. Int. Press, Cambridge, MA, 1996.

\bibitem[HW]{hw}
A.~Haydys and T.~Walpuski.
\newblock A compactness theorem for the {S}eiberg-{W}itten equation with
  multiple spinors in dimension three.
\newblock {\em Geom. Funct. Anal.}, 25(6):1799--1821, 2015.

\bibitem[JT]{jt}
A.~Jaffe and C.~Taubes.
\newblock {\em Vortices and monopoles}, volume~2 of {\em Progress in Physics}.
\newblock Birkh\"auser, Boston, Mass., 1980.
\newblock Structure of static gauge theories.

\bibitem[KW]{kw}
J.~L. Kazdan and F.~W. Warner.
\newblock Existence and conformal deformation of metrics with prescribed
  {G}aussian and scalar curvatures.
\newblock {\em Ann. of Math. (2)}, 101:317--331, 1975.

\bibitem[Moc]{mochizuki}
T.~Mochizuki.
\newblock Asymptotic behaviour of certain families of harmonic bundles on
  {R}iemann surfaces.
\newblock {\em J. Topol.}, 9(4):1021--1073, 2016.

\bibitem[MSWW]{mazzeo}
R.~Mazzeo, J.~Swoboda, H.~Weiss, and F.~Witt.
\newblock Ends of the moduli space of {H}iggs bundles.
\newblock {\em Duke Math. J.}, 165(12):2227--2271, 2016.

\bibitem[Nog]{noguchi}
M.~Noguchi.
\newblock Yang-{M}ills-{H}iggs theory on a compact {R}iemann surface.
\newblock {\em J. Math. Phys.}, 28(10):2343--2346, 1987.

\bibitem[Tau]{taubes}
C.~H. Taubes.
\newblock The zero loci of $\mathbb{Z}_2$ harmonic spinors in dimension $2$,
  $3$ and $4$.
\newblock {\em eprint arXiv:1407.6206}, 2014.

\bibitem[{Wal}]{walpuski}
T.~{Walpuski}.
\newblock {G2-instantons, associative submanifolds and Fueter sections}.
\newblock {\em eprint arXiv:1205.5350}, 2012.

\end{thebibliography}
\bibliographystyle{alphanum_n.bst}
\end{document}